\documentclass{amsart}
\usepackage{amsfonts}
\usepackage{amssymb}
\usepackage{amsmath, mathrsfs}
\usepackage{ srctex, leftidx}
 \def\sqr#1#2{{\vcenter{\vbox{\hrule
        height.#2pt \hbox{\vrule width.#2pt height#1pt \kern#2pt
          \vrule width.#2pt} \hrule height.#2pt}}}} 
\newenvironment{pf-main}{{\sc Proof of Theorem \ref{mainresult}.}\hspace{3mm}}{\qed}
\newcommand{\nc}{\newcommand}
\newtheorem{theorem}{Theorem}[section]
\newtheorem{lemma}{Lemma}[section]
\newtheorem{example}{Example}[section]
\newtheorem{corollary}{Corollary}[section]
\newtheorem{proposition}{Proposition}[section]
\newtheorem{remark}{Remark}

\nc{\cadlag}{c\`{a}dl\`{a}g } \nc{\caglad}{c\`{a}gl\`{a}d }
\nc{\ba}{\begin{array}} \nc{\ea}{\end{array}}
\nc{\be}{\begin{equation}} \nc{\ee}{\end{equation}}
\nc{\bea}{\begin{eqnarray}} \nc{\eea}{\end{eqnarray}}
\nc{\bean}{\begin{eqnarray*}} \nc{\eean}{\end{eqnarray*}}
\nc{\bu}{\bullet} \nc{\nn}{\nonumber} \nc{\cA}{{\mathcal A}}
\nc{\cB}{{\mathcal B}} \nc{\cC}{{\mathcal C}} \nc{\cD}{{\mathcal
D}} \nc{\cL}{{\mathcal L}} \nc{\cN}{{\mathcal
N}}\nc{\bbD}{\mathbb{D}}\nc{\bbZ}{\mathbb{Z}} \nc{\cG}{{\mathcal G}} \nc{\cF}{{\mathcal
F}} \nc{\cS}{{\mathcal S}} \nc{\cR}{{\mathcal
R}}\nc{\cU}{{\mathcal U}} \nc{\cH}{{\mathcal H}}
\nc{\cK}{{\mathcal K}} \nc{\cM}{{\mathcal M}} \nc{\cP}{{\mathcal
P}} \nc{\bbE}{\mathbb{E}} \nc{\bbEQ}{\mathbb{E}^{\mathbb{Q}}}
\nc{\eps}{\varepsilon}\nc{\bbU}{\mathbb{U}}
\nc{\bbEP}{\mathbb{E}_{\mathbb{P}}}\nc{\bbL}{\mathbb{L}}
\nc{\bbP}{\mathbb{P}} \nc{\bbQ}{\mathbb{Q}} \nc{\Om}{\Omega}
\nc{\om}{\omega} \nc{\bbR}{\mathbb{R}} \nc{\bbC}{\mathbb{C}}
\nc{\bfr}{\begin{flushright}} \nc{\efr}{\end{flushright}}
\nc{\dXt}{\Delta X_{t}} \nc{\dXs}{\Delta X_{s}}
\nc{\bs}{\blacksquare} \nc{\dX}{\Delta X} \nc{\dY}{\Delta Y}
\nc{\dnkx}{\left(X(T^{n}_{k})-X(T^{n}_{k-1})\right)}
\nc{\dom}{depth-of-the-market } \nc{\uar}{\uparrow}
\nc{\dar}{\downarrow}\nc{\rar}{\rightarrow}
\nc{\half}{\frac{1}{2}}\nc{\ol}{\overline}
 \nc{\hbE}{\hat{\bbE}}

\nc{\what}{\widehat} \nc{\fhat}{\what{f}}  \nc {\parx}{\frac{\partial}{\partial x}} \nc
{\parw}{\frac{\partial}{\partial w}} \nc
{\parww}{\frac{\partial^2}{\partial w^2}}
\def\rar{\rightarrow}
\def\lar{\leftarrow}
\def\dar{\downarrow}

\nc{\chf}{\mbox{$\mathbf1$}} \nc{\eid}{\stackrel{d}{=}}
\numberwithin{equation}{section}
\begin{document}
\title{On certain integral functionals of squared Bessel processes}
        \author{Umut \c{C}etin}
 \address{London School of Economics and Political Science,
        Department of Statistics, Columbia House, Houghton Street, London WC2A 2AE}
\date{\today}
\maketitle
\begin{abstract} For a squared Bessel process, $X$, the Laplace transforms of joint laws of $(U,
  \int_0^{R_y}X_s^p\,ds)$ are studied where $R_y$ is the
  first hitting time of $y$ by $X$ and $U$ is a random variable measurable with respect to the history of $X$ until $R_y$. A subset of these results are
  then used to solve the associated small ball problems for $\int_0^{R_y}X_s^p\,ds$ and determine a
  Chung's law of iterated logarithm.  $\left(\int_0^{R_y}X_s^p\,ds\right)$ is also considered as a purely
  discontinuous increasing Markov process and its infinitesimal
  generator is found. The findings are then used to price a class of
  exotic derivatives on interest rates and determine the asymptotics
  for the prices of some put options that are only slightly in-the-money.
\end{abstract}.

\section{Introduction}
Let $X$ be a squared Bessel process which is the unique strong solution to
\[
dX_t=2 (\nu +1)\,dt +2\sqrt{X_t}\,dB_t,
\]
where $\nu \geq -1$ is a real constant and $B$ is a standard Brownian
motion.  Letting $\delta=2(\nu+1)$, $X$ is  called a
$\delta$-dimensional squared Bessel process. We will denote
such a process with $X_0=z$ by $BESQ^{\delta}(z)$ and $\delta$ and $\nu$ will be
related by $\delta=2(\nu+1)$ throughout the text. In this paper we are interested in the integral functional
\be \label{d:Sigma}
\Sigma^{\delta}_{p,z,y}:=\int_0^{R_y}X_s^p
    \,ds,
\ee
where $p>-1$ and $R_y:=\inf\{t\geq 0: X_t=y\}$ for $y \in [0, \infty)$ and $X$ is $BESQ^{\delta}(z)$. (In the
sequel, we will write $R^{\delta}_y$ only if we need to specify the
dimension to avoid ambiguity.) 

Squared Bessel processes have found wide applications  especially in
Finance Theory, see Chapter 6 in \cite{cjy} for a recent account. They can, e.g.,
be used to model interest rates in a {\em Cox-Ingersoll-Ross} framework. In the above setting, if   $X^p$
models the spot interest rates, then $\exp\left
  (\Sigma^{\delta}_{p,z,y}\right)$ refers to the cumulative interest
until the spot rate hits the barrier $y^p$. As such, this random
variable is related to  certain  {\em exotic options} on interest
rates (see \cite {dl} for some formulae regarding barrier options in a similar framework). Bessel processes also appear often in the study of {\em financial bubbles} since $1/\sqrt{X}$ is the prime example of a continuous (strict) local martingale when $X$ is a $BESQ^3$ (see, e.g., \cite{my}, \cite{sp} and \cite{sp1} for how  strict local martingales, and in particular Bessel processes, appear in mathematical studies of bubbles).

In Section 2 we will determine the joint  law of $(U, \Sigma^{\delta}_{p,z,y})$ by
martingale methods, where $U$ is a random variable measurable with
respect to the evolution of $X$ until  $R_y$. In particular we will obtain the joint
distributions of  $(R_y, \Sigma^{\delta}_{p,z,y})$  and  $(\max_{t\leq
  R_y}X_t, \Sigma^{\delta}_{p,z,y})$. As a by-product of our findings,
if  $\frac{|\nu|}{p+1}=\frac{1}{2}$ , we have a remarkable characterisation
of the conditional law of $\Sigma^{\delta}_{p,z,y}$  given that the maximum (resp. minimum) of $X$ at
$R_y$ is below (resp. above) a fixed level in terms of the first
hitting time distributions of a $3$-dimensional Bessel process when $z \geq y$ (resp.
 $z\leq y$).

We will use the results of Section 2 in order to study {\em small
  ball probabilities} for $\Sigma^{\delta}_{p,z,y}$ in Section 3. Solving the small
ball problem for $\Sigma^{\delta}_{p,z,y}$ amounts to finding the
asymptotic behaviour of $-\log
\mbox{Prob}(\Sigma^{\delta}_{p,z,y}<\eps)$ as $\eps\rar 0$. We will
then use this asymptotic form to determine a {\em law of iterated
  logarithm} for $(\Sigma^{\delta}_{p,0,y})_{y \geq 0}$ as $y \rar
\infty$.

Section 4 will analyse $(\Sigma^{\delta}_{p,0,y})_{y
  \geq 0}$ as a Markov process indexed by $y$ and compute its
infinitesimal generator when $\nu \geq 0$. We will also consider the
process $Z^{\delta}$ which is obtained via a `time reversal' from
$(\Sigma^{\delta}_{p,0,y})_{y \geq 0}$. More precisely, we will find
the generator of $Z^{\delta}$  defined by
\[
Z^{\delta}_x=\int_{L_{1-x}}^{L_1}X_s^p\,ds \qquad \forall x \in[0,1),
\]
where $L_x:=\sup\{t\geq 0:X_t=x\}$. In particular, we will obtain that
$Z^4$ is identical in law to an increasing family of hitting times of
a linear Brownian motion.

Finally, in Section 5 we will apply our findings to the pricing of some
exotic derivatives on interest rates. The small ball probabilities
will be used to find asymptotic behaviour of some put options with
small strikes, the options that are only slightly {\em in-the-money}.
\section{Preliminaries} \label{s:2}
Let $(\Om, \cF, (\cF_t)_{t \geq 0}, \bbP)$ be a stochastic base
where $\cF$ is completed with the $\bbP$-null sets. Let  $X$ be an $\bbR_+$-valued
semimartingale which is the unique strong solution to
\be \label{e:sdeX}
dX_t=2 (\nu +1)\,dt +2\sqrt{X_t}\,dB_t,
\ee
where $\nu \geq -1$ is a real constant, $B$ is a standard Brownian
motion.
%, and $\bbP(X_0\in dx)=\mu(dx)$, where $\mu$ is a probability
%measure on $\bbR_+$.

Let  $Q^{\delta}_z$ be  the  measure on the path space, i.e. $C([0,\infty), [0,\infty))$, induced by $X$ starting at $z$. It is well-known (see Section 1 of Chapter XI of \cite{ry})
that for $\nu  \geq 0$ the set $\{0\}$ is polar, otherwise it is
reached a.s.. Moreover, the process is transient for $\nu>0$ and recurrent otherwise. We will denote the first hitting time of $0$ for $X$
with $R$. The scale function, $s^{\nu}$, for $BESQ^{\delta}$ is given
by
\[
s^{\nu}(x)=-x^{-\nu}\qquad \mbox{for } \nu>0, \qquad s^0(x)=\log x,
\qquad
s^{\nu}(x)=x^{-\nu}\qquad \mbox{for } \nu \in [-1,0).
\]
We refer the reader to \cite{ry} and \cite{gy} for a comprehensive study of Bessel processes and relevant bibliography.

In subsequent computations we will follow a Feynman-Kac type approach as in, e.g., \cite{jpy}.
\begin{lemma} \label{l:lmart} Let $p> -1, \lambda >0$ and suppose
  that $u \in C^2$, solves the
  following ordinary differential equation (ODE):
\be \label{e:ode}
x^2 y'' + x y' -y [\nu^2 +\lambda x^{2(p+1)}]=0,
\ee
and is strictly positive on $(0, \infty)$.
Then,  $(M^{(u)}_{t \wedge R})_{t \geq 0}$ is a local martingale where
\[
M^{(u)}_t:=u(\sqrt{X_{t}})X_{t}^{-\frac{\nu}{2}}\exp\left(-\frac{\lambda}{2}\int_0^{t} X_s^p\,
  ds\right).
\]
In particular,
\[
\chf_{[t < R]}dM^{(u)}_t=\chf_{[t < R]}M^{(u)}_t\left(\frac{u'(\sqrt{X_t})}{u(\sqrt{X_t})}-\frac{\nu}{\sqrt{X_t}}\right)dB_t.
\]
\end{lemma}
\begin{proof} If we let $w(x):=u(\sqrt{x})x^{-\frac{\nu}{2}}$, it is
  easily seen that $w$ solves
\be \label{e:odew}
2x w'' + 2(\nu+1) w' -\frac{\lambda}{2}x^p w=0.
\ee
Thus, on $[t <R]$
\[
dM^{(u)}_t= 2 M^{(u)}_t
\frac{w'}{w}(X_t)\sqrt{X_t}dB_t=M^{(u)}_t\left(\frac{u'(\sqrt{X_t})}{u(\sqrt{X_t})}-\frac{\nu}{\sqrt{X_t}}\right)dB_t.
\]
\end{proof}
The solutions to (\ref{e:ode}) can easily be determined via the
modified Bessel functions, $I_{\alpha}$ and $K_{\alpha}$, of the first
and second kind. We next summarise some properties of these functions and we refer the reader to Section 3.7 of \cite{tbf} or Section 9.6.1 of \cite{AS} for further results and proofs.  For $\alpha \geq 0$, $I_\alpha$
(resp. $K_{\alpha}$) is a positive and increasing (resp. decreasing) solution to
\be \label{e:ModBes}
x^2 y'' + x y' -(x^2 + \alpha^2) y=0
\ee
when $x$ is restricted to  $(0,\infty)$. Moreover, $I_{\alpha}$ has  the following series representation:
\be \label{e:mb1def}
I_{\alpha}(x)=\sum_{m=0}^{\infty} \frac{1}{m! \Gamma(m+\alpha
  +1)}\left(\frac{x}{2}\right)^{2 m +\alpha}.
  \ee
The above expansion can be used to define $I_{\alpha}$ for $\alpha <0$ when $\alpha$ is not an integer less than $-1$ since the Gamma function is well defined and finite at non-integer negative values. For $\alpha \in \bbZ \cap (-\infty, 0)$ one can define it by taking limits and it turns out that for such an $\alpha$ $I_{-\alpha}=I_{\alpha}$. It is also  a simple matter to check that $I_{\alpha}$ satisfies (\ref{e:odew}) for $\alpha<0$ as well. Moreover,
\be\label{e:mb2def}
K_{\alpha}(x)=\frac{\pi}{2}\frac{I_{-\alpha}(x)-I_{\alpha}(x)}{\sin(\alpha
  \pi)}.
\ee
The above identity in particular entails $K_{\alpha}=K_{-\alpha}$. 

The asymptotic behaviour of the modified Bessel functions can be described in terms of known functions:
\be \label{e:MBasymp}
\ba{ll}
I_{\alpha}(x)  \sim  \frac{\left(\frac{x}{2}\right)^{\alpha}}{\Gamma(\alpha+1)} \mbox{ as } x \rar 0,\, \alpha \neq -1, -2, \ldots; & I_{\alpha}(x)  \sim \frac{e^x}{\sqrt{2 \pi x}} \mbox{ as } x \rar \infty. \\
K_{\alpha}(x) \sim  \frac{\Gamma(\alpha)}{2}\left(\frac{x}{2}\right)^{-\alpha} \mbox{ as } x \rar 0, \alpha >0; & K_{\alpha}(x)  \sim \sqrt{\frac{\pi}{2 x}}e^{-x} \mbox{ as } x \rar \infty.\\
K_{0}(x) \sim - \log x \mbox{ as } x \rar 0. & 
\ea
\ee
%There exist also integral representations (see p.~172 of \cite{tbf}) as follows:
%\bea
%I_{\alpha}(x) &=& \frac{\left(\frac{1}{2} x\right)^{\alpha}}{\Gamma\left(\alpha+\frac{1}{2}\right)\Gamma \left(\frac{1}{2}\right)}\int_{-1}^1(1-t^2)^{\alpha-\frac{1}{2}}e^{-xt}\, dt, \qquad \alpha >-\frac{1}{2}; \label{e:intI}\\
%K_{\alpha}(x)&=&\frac{\Gamma \left(\frac{1}{2}\right)\left(\frac{1}{2} x\right)^{\alpha}}{\Gamma\left(\alpha+\frac{1}{2}\right)}\int_1^{\infty}e^{-xt}(t^2-1)^{\alpha-\frac{1}{2}}\,dt, \qquad \alpha >-\frac{1}{2} \label{e:intK}.
%\eea
Then,  it is easy to check that the solutions to (\ref{e:ode}) is of
the form
\be
C_1
K_{\frac{|\nu|}{p+1}}\left(\frac{1}{p+1}\sqrt{\lambda}x^{p+1}\right) +
C_2
I_{\frac{|\nu|}{p+1}}\left(\frac{1}{p+1}\sqrt{\lambda}x^{p+1}\right),
\ee
where $C_1$ and $C_2$ are arbitrary constants.
\begin{remark} \label{r:monotone} It is worth to observe some further monotonicity properties regarding modified Bessel functions that will be useful in the sequel. 

First consider $x^{-\nu} K_{\frac{|\nu|}{p+1}}\left(\frac{1}{p+1}\sqrt{\lambda}x^{p+1}\right)$ for $\nu\in \bbR$ and $p>-1$. We will now see that this function is decreasing. Indeed, differentiation with respect to $x$ yields
\[
-\nu x^{-\nu-1} K_{\frac{\nu}{p+1}}\left(\frac{1}{p+1}\sqrt{\lambda}x^{p+1}\right) + \sqrt{\lambda}x^{-\nu+p} K_{\frac{\nu}{p+1}}'\left(\frac{1}{p+1}\sqrt{\lambda}x^{p+1}\right).
\]
When $\nu \geq 0$, utilising the recurrence elation (cf. Section 3.7 in \cite{tbf})
\[
K'_{\alpha}(x)=\frac{\alpha}{x}K_{\alpha}(x)-K_{\alpha+1}(x), \; \alpha \in \bbR,
\]
one can see that the above derivative is negative since $K_{\alpha}$ is positive for all $\alpha \in \bbR$. If $\nu <0$, the other recurrence relation, i.e. 
\[
K'_{\alpha}(x)=-\frac{\alpha}{x}K_{\alpha}(x)-K_{\alpha-1}(x),  \; \alpha \in \bbR,
\]
yields the conclusion. 

Similarly, the function $x^{-\nu} I_{\frac{\nu}{p+1}}\left(\frac{1}{p+1}\sqrt{\lambda}x^{p+1}\right)$, has a positive derivative, i.e.  is increasing, whenever $p>-1$ and $\frac{\nu}{p+1}>-2$. This follows  from the recurrence relation 
\[
I'_{\alpha}(x)=\frac{\alpha}{x}I_{\alpha}(x)+ I_{\alpha+1}(x), \; \alpha \in \bbR,
\]
and the positivity of $I_{\alpha}$ for $\alpha>-1$. Note that $\frac{\nu}{p+1}+1>-1$ under our hypotheses. 
\end{remark}
We now return to determining the joint law of $(U,
\Sigma^{\delta}_{p,z,y})$ for arbitrary positive
$\cF_{R_y}$-measurable random variables $U$, where
$\Sigma^{\delta}_{p,z,y}$ is as
defined in (\ref{d:Sigma}). We will analyse the cases of negative and
positive $\nu$ separately.
\subsection{The case $\nu \in (-1,0)$}
\begin{theorem} \label{t:numinus} Suppose that $\nu \in (-1,0), \, p>-1$ and let
  $u_0(x):=K_{\frac{-\nu}{p+1}}\left(\frac{1}{p+1}\sqrt{\lambda}x^{p+1}\right)$.
  Then, 
  \[
0<\lim_{x \rar 0} u_0(\sqrt{x})x^{-\frac{\nu}{2}}< \infty, \qquad \lim_{x
  \rar \infty} u_0(\sqrt{z})x^{-\frac{\nu}{2}}< \infty.
\]
Consequently, $(M^{(u_0)}_{t \wedge R})_{t \geq 0}$ is a strictly positive bounded
  martingale with
\[
\chf_{[t < R]}dM^{(u_0)}_t=\chf_{[t < R]}M^{(u_0)}_t\left(\frac{u_0'(\sqrt{X_t})}{u_0(\sqrt{X_t})}-\frac{\nu}{\sqrt{X_t}}\right)dB_t.
\]
\end{theorem}
\begin{proof}Note that 
\[
\lim_{x\rar 0} u_0(\sqrt{x})x^{-\frac{\nu}{2}}= \left(\frac{\sqrt{\lambda}}{p+1}\right)^{\frac{\nu}{p+1}}\lim_{x \rar 0} x^{ -\frac{\nu}{p+1}}K_{-\frac{\nu}{p+1}}(x)=\left(\frac{2\sqrt{\lambda}}{p+1}\right)^{\frac{\nu}{p+1}}\frac{\Gamma(-\frac{\nu}{p+1})}{2}
\]
in view of the asymptotic relations from (\ref{e:MBasymp}). Asymptotic behaviour of $K_{-\frac{\nu}{p+1}}$ for large $x$ yields $u_0(\sqrt{x})x^{-\frac{\nu}{2}}$ has a finite limit at infinity, hence the boundedness of $(M^{(u_0)}_{t
  \wedge R})_{t \geq 0}$ 
in view of Lemma \ref{l:lmart}. Strict positivity of $K_{\alpha}$ on $(0, \infty)$ for all $\alpha$ completes the proof of that this martingale is strictly positive. 
\end{proof}
It is  well-known that (see, e.g. Section 2.8 in \cite{py})  for $\nu>-1$
\bea
Q^{\delta}_z
\left[\exp\left(-\frac{\lambda}{2}R_y\right)\right]&=&\frac{z^{-\frac{\nu}{2}}K_{\nu}(\sqrt{\lambda
    z})}{y^{-\frac{\nu}{2}}K_{\nu}(\sqrt{\lambda y})}, \qquad y \leq
z; \label{e:LRyx}\\
&=&\frac{z^{-\frac{\nu}{2}}I_{\nu}(\sqrt{\lambda
    z})}{y^{-\frac{\nu}{2}}I_{\nu}(\sqrt{\lambda y})}, \qquad y \geq
z. \label{e:LRxy}
\eea
The above formulae are still valid when $\nu \geq 0$ and note that $z^{-\frac{\nu}{2}}I_{\nu}(\sqrt{\lambda z})$ (resp. $ z^{-\frac{\nu}{2}}K_{\nu}(\sqrt{\lambda z})$) is increasing (resp. decreasing) in view of Remark \ref{r:monotone}.

Since $R<\infty$, a.s. when  $\nu < 0$,   the following is a straightforward corollary
to the theorem above for $\nu <0$ and $y < z$.
\begin{corollary} \label{c:jointl} Let $u_0$ be the function defined in Theorem
  \ref{t:numinus} and suppose that $\nu \in (-1,0), \, p>-1$ and $y < z$. If $U$ is
  $\cF_{R_y}$-measurable, then for $r \geq 0$,
\[
Q^{\delta}_z \left[\exp\left(-r U -\frac{\lambda}{2}
    \Sigma^{\delta}_{p,z,y}\right)\right]=\frac{u_0(\sqrt{z})}{u_0(\sqrt{y})}\left(\frac{z}{y}\right)^{-\frac{\nu}{2}}
P^{\delta, u_0}_z\left[\exp\left(-r U\right)\right],
\]
where $P^{\delta, u_0}_z$ is defined by $\frac{dP^{\delta, u_0
  }_z}{dQ^{\delta}_z}=M^{(u_0)}_{R}$. Moreover, under $P^{\delta,
  u_0}_z$, $X$ solves
\be \label{e:Xamc}
dX_t=2\left(\frac{u_0'(\sqrt{X_t})\sqrt{X_t}}{u_0(\sqrt{X_t})}+1\right)\,dt
+2 \sqrt{X_t}\,d\beta_t, \qquad t \leq R,
\ee
for some $P^{\delta, u_0}_z$-Brownian motion $\beta$.
\end{corollary}
\begin{proof}
Since $(M^{(u_0)}_{t \wedge R})_{t \geq 0}$ is a strictly positive martingale due to Theorem \ref{t:numinus}, $Q^{\delta}_z$ is an equivalent probability measure. The fact that $X$ solves (\ref{e:Xamc}) follows from an application of Girsanov's theorem. 
\end{proof}
\begin{remark} By taking $U \equiv 0$ Corollary \ref{c:jointl} yields the law of $\Sigma^{\delta}_{p,z,y}$ for $y < z$. Comparing this Laplace transform with (\ref{e:LRyx}) shows that $\Sigma^{\delta}_{p,z,y}\eid R^{\delta^{\ast}}_{y^{\ast}}$ where $\delta^{\ast}=2(\frac{\nu}{1+p}+1), \, y^{\ast}=\frac{y^{1+p}}{(1+p)^2}$ and $R^{\delta^{\ast}}_{y^{\ast}}$ is the first hitting time of $y^{\ast}$ for some $BESQ^{\delta^{\ast}}\left(\frac{z^{1+p}}{(1+p)^2}\right)$. One can check by comparing the Laplace transforms that we will obtain later in this section that  this equality in law would be valid for $y\geq z$, too. Moreover, the same identities in distribution will hold for $\nu \geq 0$, too. These facts also follow from the time-change result given in  Proposition XI.1.11 in \cite{ry}.
\end{remark}
Recall that if $X$ is a one-dimensional regular diffusion on an interval $(l_0,l_1)$ defined by its infinitesimal generator, $\cA$, where
\[
\cA= \half \sigma^2(x)\frac{d^2}{dx^2} + b(x) \frac{d}{dx}
\]
for some locally bounded functions $\sigma$ and $b$ such that $\sigma>0$ on the open interval $(l_0,l_1)$, then for any $r>0$ there exists  a  positive, and strictly decreasing (resp. increasing) function $\Phi$ (resp. $\Psi$), which solves
\[
\cA u= r \Phi \; (\mbox{resp. } r \Psi)
\]
on $(l_0, l_1)$ satisfying certain boundary requirements depending on the nature of the behaviour of the diffusion  near $l_0$ and $l_1$. Moreover, any other solution of this equation with the above positivity and monotonicity assumptions is a fixed multiple of $\Phi$ (resp. $\Psi$) (see, e.g., Proposition V.50.3 in \cite{rw}). Then,  if $R_y$ is the first hitting time of $y$, 
\bean
P_z\left[\exp\left(- r R_y\right)\right]&=&\frac{\Phi(z)}{\Phi(y)}, \qquad y \leq
z; \\
&=&\frac{\Psi(z)}{\Psi(y)}, \qquad y \geq
z, 
\eean
where $P_z$ is the law of the diffusion that started at $z$ at $t=0$.

Thus,  if $U=R_y$, we obtain the following in view of the above discussion.
\begin{corollary} \label{c:ode-sdeK} Let $u_0$ be the function defined in Theorem
  \ref{t:numinus} and suppose that $\nu \in (-1,0)$ and $y < z$. Then for $r \geq 0$,
\[
Q^{\delta}_z\left[\exp\left(-r R_y -\frac{\lambda}{2}
    \Sigma^{\delta}_{p,z,y}\right)\right]=\frac{u_0(\sqrt{z})}{u_0(\sqrt{y})}\left(\frac{z}{y}\right)^{-\frac{\nu}{2}}
\frac{\Phi(z)}{\Phi(y)},
\]
where $\Phi$ is a positive and decreasing solution of
\be \label{e:Phi}
2x v'' + 2
\left(\frac{u_0'(\sqrt{x})\sqrt{x}}{u_0(\sqrt{x})}+1\right)v'=rv.
\ee
\end{corollary}
\begin{proof}
The ODE in (\ref{e:Phi}) corresponds to the diffusion, $X$, which follows
\be \label{e:sdec21}
dX_t=2\left(\frac{u_0'(\sqrt{X_t})\sqrt{X_t}}{u_0(\sqrt{X_t})}+1\right)\,dt
+2 \sqrt{X_t}\,d\beta_t,
\ee
where $\beta$ is a Brownian motion, if a solution exists. Using the recursive relation 
\[
K'_{\alpha}(x)=-\frac{\alpha}{x}K_{\alpha}(x)-K_{\alpha-1}(x),
\]
we obtain
\[
\frac{u_0'(\sqrt{x})\sqrt{x}}{u_0(\sqrt{x})}=\nu -\sqrt{\lambda}x^{\frac{p+1}{2}}\frac{K_{\frac{-\nu-p-1}{p+1}}\left(\frac{1}{p+1}\sqrt{\lambda}x^{\frac{p+1}{2}}\right)}{K_{\frac{-\nu}{p+1}}\left(\frac{1}{p+1}\sqrt{\lambda}x^{\frac{p+1}{2}}\right)} \leq \nu.
\]
This shows that the drift of the above SDE is less than that of the SDE solved by a $BESQ^{\delta}$. Thus, when a solution exists, it never explodes due to the standard comparison results for SDEs.  Moreover, using the asymptotic relations for $K$ as $x$ tends to $0$, we can directly verify that 
\[
\lim_{x\rar 0} \frac{u_0'(\sqrt{x})\sqrt{x}}{u_0(\sqrt{x})}=\nu.
\]
Thus, the solution is pushed towards the interior of $(0,\infty)$ as soon as it hits $0$ since the diffusion coefficient vanishes at $0$. Finally, the existence of a solution taking values in $[0, \infty)$ follows easily since the drift term is locally Lipschitz as $K_{\alpha}$ is strictly positive on $[0, \infty)$ and $K_{\alpha}$ is twice continuously differentiable for all $\alpha$. The solution is also unique due to drift coefficient being locally Lipschitz.

Recall that in view of Corollary \ref{c:jointl}  $X$ solves (\ref{e:sdec21}) until the first time it hits $0$. Since it cannot hit $0$ before hitting $y$, the formula is a direct consequence of the discussion preceding the corollary.
\end{proof}
Corollary \ref{c:jointl} also allows us to compute the law of
$\Sigma^{\delta}_{p,z,y}$ on the event that a certain boundary is yet
to be reached via the intimate relationship between the scale functions of diffusions and their exit probabilities (see Sect.3 of Chap. VII in \cite{ry} for details).
\begin{corollary} \label{c:jmS} Let $u_0$ be the function defined in Theorem
  \ref{t:numinus} and suppose that $\nu \in (-1,0), \, p>-1$ and $y <z$. Then, a
  scale function of the diffusion defined by (\ref{e:Xamc}) is
\be \label{e:stilde0}
\tilde{s_0}(x)= \int_1^x\frac{1}{y u_0^2(\sqrt{y})}dy, \qquad x\geq 0.
\ee
Thus, for any $a >x$
\[
Q^{\delta}_z\left[\chf_{[R_a >R_y]}\exp\left( -\frac{\lambda}{2}
    \Sigma^{\delta}_{p,z,y}\right)\right]= \frac{u_0(\sqrt{x})}{u_0(\sqrt{y})}\left(\frac{x}{y}\right)^{-\frac{\nu}{2}} \frac{\tilde{s_0}(x)-\tilde{s_0}(a)}{\tilde{s_0}(y)-\tilde{s_0}(a)}.
\]
\end{corollary}
\begin{proof} The representation of the scale function is due to
  the well-known formulas for the solutions of SDEs, see, e.g.,
  Exercise VII.3.20 in \cite{ry}.  Note that the function is
  well-defined at $x=0$. Indeed, it follows from Theorem \ref{t:numinus} that $\lim_{y\rar
    0}u_0(\sqrt{y})y^{-\frac{\nu}{2}}>0$. Thus, $ y
  u_0^2(\sqrt{y}) > C y^{1+\nu}$ for some $C>0$ for sufficiently small $y$. Since $y^{-(1+\nu)}$
  is integrable for $\nu \in [-1,0)$, the claim holds.

 The second assertion follows from
  Corollary \ref{c:jointl} after taking $r=1$ and $U=\log  \chf_{[R_a
    >R_y]}$  via the defining property of scale functions, see
  Definition VII.3.3  in \cite{ry}.
\end{proof}
\begin{remark} The above result in fact gives us the joint law of
  $(\max_{t \leq R_y}X_t, \Sigma^{\delta}_{p,z,y})$. Indeed, for any
  $a \geq z > y$
\[
[\max_{t \leq R_y}X_t< a]=[R_a >R_y].
\]
\end{remark}

Since $K_{\frac{1}{2}}(x)=\sqrt{\frac{\pi}{2x}}e^{-x}$, we have more
  explicit formulas when $\frac{\nu}{p+1}=-\frac{1}{2}$.

\begin{corollary} \label{c:bht-} Suppose that $\frac{\nu}{p+1}=-\frac{1}{2}$ and $p>-1$. Then, for
  $y \leq z$ we have the following:
\begin{itemize}
\item[i)] \[
Q^{\delta}_z\left[\exp\left(-r R_y -\frac{\lambda}{2}
    \Sigma^{\delta}_{p,z,y}\right)\right]=\exp\left(\sqrt{\lambda}\frac{z^{-\nu}-y^{-\nu}}{2 \nu}\right)
\frac{\Phi(z)}{\Phi(y)},
\]
where $\Phi$ is a positive and decreasing solution of
\be \label{e:Phi0}
2x v'' + 2
\left(\nu +1 -\sqrt{\lambda} x^{-\nu}\right)v'=rv
\ee
on $(0, \infty)$.
\item[ii)] The function $\tilde{s_0}$ is, up to an affine transformation, given by $\exp\left(-\sqrt{\lambda}\frac{x^{-\nu}}{\nu}\right)$.
\item[iii)] For $a >z$,
\bea
Q^{\delta}_z\left[\chf_{[R_a >R_y]}\exp\left( -\frac{\lambda}{2}
    \Sigma^{\delta}_{p,z,y}\right)\right]&=& \exp\left(\sqrt{\lambda}\frac{z^{-\nu}-y^{-\nu}}{2 \nu}\right)\frac{\exp\left(-\sqrt{\lambda}\frac{z^{-\nu}}{ \nu}\right)-\exp\left(-\sqrt{\lambda}\frac{a^{-\nu}}{ \nu}\right)}{\exp\left(-\sqrt{\lambda}\frac{y^{-\nu}}{ \nu}\right)-\exp\left(-\sqrt{\lambda}\frac{a^{-\nu}}{ \nu}\right)}\nn
    \\
    &=&\frac{\sinh\left(-\sqrt{\lambda}\frac{a^{-\nu}-z^{-\nu}}{
          2\nu}\right)}{\sinh\left(-\sqrt{\lambda}\frac{a^{-\nu}-y^{-\nu}}{
          2\nu}\right)}. \label{e:lM<a}
\eea
\end{itemize}
\end{corollary}
Note that the expression in (\ref{e:lM<a}) yields
\[
Q^{\delta}_z\left[\exp\left( -\frac{\lambda}{2}
    \Sigma^{\delta}_{p,z,y}\right)\bigg| R_a >R_y\right]=\frac{a^{-\nu} -y^{-\nu}}{a^{-\nu} -z^{-\nu}}\frac{\sinh\left(-\sqrt{\lambda}\frac{a^{-\nu}-z^{-\nu}}{
          2\nu}\right)}{\sinh\left(-\sqrt{\lambda}\frac{a^{-\nu}-y^{-\nu}}{
          2\nu}\right)}
\]
using the scale function of $X$ under $Q^{\delta}_z$. Comparing this
with  (\ref{e:LRxy}) for $\nu=1/2$ gives the following since $I_{\half}(x)=\sqrt{\frac{2\pi}{x}}\sinh(x)$.
\begin{corollary} \label{c:sigcl-} Suppose that $\frac{\nu}{p+1}=-\frac{1}{2}$, $\nu \in (-1,0)$, and $p>-1$. Then, for
  $y \leq z <a$ we have that the law of $\Sigma^{\delta}_{p,z,y}$
  conditioned on the event $[R_a>R_y]$ is that of  the first hitting
  time of $(a^{-\nu}-y^{-\nu})^2/4\nu^2$ by a $3$-dimensional squared
  Bessel process started at $(a^{-\nu}-z^{-\nu})^2/4\nu^2$.
\end{corollary}
Note that, since $\lim_{a \rar 0}=\frac{1-x^a}{a}=-\log x$, when $y>0$, we obtain that the above conditional laws converge as $\nu \rar 0$ (and, thus, as $p \rar -1$) to that of the first hitting time of $(\log\sqrt{a}-\log\sqrt{y})^2$ by a $3$-dimensional squared Bessel process started at $(\log\sqrt{a}-\log\sqrt{z})^2$. This can be viewed as the analogous statement of the above corollary when $\nu=0$ and $p=-1$.

  Next we look at the case when $y> z \geq 0$. Observe that the function $u_1$ as defined in the theorem below is still well defined and finite at $z=0$ in view of, e.g.,  the series representation of $I_{\alpha}$ in (\ref{e:mb1def}).  \begin{theorem} \label{t:numinus2}  Suppose  that  $p \geq 0$, $y\geq z$, and $\nu \in (-1,0)$. Let
  $u_1(x):=I_{\frac{\nu}{p+1}}\left(\frac{\sqrt{\lambda}}{p+1}
    x^{p+1}\right)$. Then,  $(M^{(u_1)}_{t \wedge R_y})_{t \geq 0}$ is a bounded
  martingale with
\[
\chf_{[t < R_y]}dM^{(u_1)}_t=\chf_{[t < R_y]}M^{(u_1)}_t\left(\frac{u_1'(\sqrt{X_t})}{u_1(\sqrt{X_t})}-\frac{\nu}{\sqrt{X_t}}\right)dB_t,
\]
for any $ y \geq 0$.
\end{theorem}
\begin{proof} Note  that $[R < R_y]$ has a positive
  probability. Thus, we have to pay attention to the behaviour of
  $w(x)=u_1(\sqrt{x})x^{-\frac{\nu}{2}}$ at $x=0$. Observe that under
  our assumptions, $\frac{\nu}{p+1}>-1$, thus it follows from the
  series representation of $I_{\alpha}$ that $w(0)>0$ and is
  finite since $I_{\alpha}(1)< \infty$ for any $\alpha$. Next, we will show that $w$ has an absolutely continuous
  derivative over $[0, \infty)$. Using the recurrence
  relations (see Section 3.7 in \cite{tbf})
\bean
I_{\alpha -1}(x)+I_{\alpha+1}(x)&=&2 I'_{\alpha}(x), \mbox{ and}\\
\frac{x}{2}\left(I_{\alpha -1}(x)-I_{\alpha+1}(x)\right)&=&\alpha
I_{\alpha}(x),
\eean
we obtain that
\be \label{e:recI}
I'_{\alpha}(x)=I_{\alpha+1}(x)+\frac{\alpha}{x}I_{\alpha}(x).
\ee
Using this identity it follows from direct calculations that
\[
w'(x)=\frac{\sqrt{\lambda}}{2}x^{\frac{p-\nu-1}{2}}I_{\gamma}\left(\frac{\sqrt{\lambda}}{p+1}
    x^{\frac{p+1}{2}}\right),
\]
where $\gamma=1+\frac{\nu}{p+1}$. Since the leading term of $I_{\gamma}\left(\frac{\sqrt{\lambda}}{p+1}x^{\frac{p+1}{2}}\right)$ as $x \rar 0$ is
  $x^{\frac{\nu+p+1}{2}}$ , we see that $\lim_{x \rar
    0}w'(x)=0$ when $p > 0$.  Therefore, we obtain immediately from the ODE (\ref{e:odew}) that when $p>0$  $\lim_{x\rar 0} xw''(x)=0$ for $\nu \in (-1,0)$. On the other hand, when $p=0$ and $\nu>-1$,
\[
\frac{w'(x)}{w(x)}=\frac{\lambda}{2} \frac{1}{\sqrt{\lambda x}}\frac{I_{v+1}(\sqrt{\lambda x})}{I_v(\sqrt{\lambda x})}
\]
by another application of (\ref{e:recI}). Thus, in view of the asymptotics  of $I_{\alpha}$ as $x \rar 0$
\bean
\lim_{x \rar 0}\frac{w'(x)}{w(x)}&=& \frac{\lambda}{2}\lim_{x \rar 0}\frac{I_{v+1}( x)}{x I_v(x)} \nn \\
&=&\frac{\lambda}{4}  \frac{\Gamma(\nu+1)}{\Gamma(\nu+2)}=\frac{\lambda}{4(\nu+1)} .
\eean
Consequently, $\lim_{x \rar 0} 2(\nu+1)w' -\frac{\lambda}{2}w=0$ since $w(x)>0$ for all $x \geq 0$. Again, it follows from the ODE (\ref{e:odew})  that
 $\lim_{x\rar 0} xw''(x)=0$. However, this condition implies that $\int_0^x w''(y)\,dy$ exists and is finite. Since this
integral equals $w'(x)-w'(0)$ for any $ x \in [0, \infty)$, we conclude that
$w'$ is absolutely continuous on $[0, \infty)$ and $w'(x)=w'(0)+ \int_0^x
w''(y)\,dy$ for any $x \in [0, \infty)$. Then, in view of Problem
3.7.3 in \cite{ks} we immediately deduce that
\bean
w(X_t)&=&w(X_0)+\int_0^t 2 w'(X_s) \sqrt{X_s}\,dB_s +\int_0^t \left\{2
  w'(X_s)(\nu+1)+ 2 w''(X_s)X_s\right\}ds\\
&=&\int_0^t 2 w'(X_s) \sqrt{X_s}\,dB_s +\frac{\lambda}{2}\int_0^t
X_s^p w(X_s)\, ds.
\eean
A simple application of integration by parts formula now shows that
$M^{(u_1)}$ is a martingale with the claimed representation.
\end{proof}
\begin{remark}
Observe that $u_1(x) x^{-\nu}$ is still monotone in view of Remark \ref{r:monotone} since under our hypothesis $\frac{\nu}{p+1}>-2$.
\end{remark}

\begin{corollary} \label{c:k-1/2} Let $u_1$ be the function defined in Theorem
  \ref{t:numinus2} and suppose that the hypotheses therein hold.  Then, we have the following for all $z\leq y$:
  \begin{itemize}
  \item[i)] If  $r \geq 0$ and $U$ is $\cF_{R_y}$-measurable,
\[
Q^{\delta}_z\left[\exp\left(-r U -\frac{\lambda}{2}
    \Sigma^{\delta}_{p,z,y}\right)\right]=\frac{u_1(\sqrt{z})}{u_1(\sqrt{y})}\left(\frac{z}{y}\right)^{-\frac{\nu}{2}}
P^{\delta, u_1}_{z,y}\left[\exp\left(-r U\right)\right],
\]
where $P^{\delta, u_1}_{z,y}$ is defined by $\frac{dP^{\delta, u_1
  }_{z,y}}{dQ^{\delta}_z}=M^{(u_1)}_{R_y}$. Moreover, under $P^{\delta,
  u_1}_{z,y}$, $X$ solves
\be \label{e:Xamc2}
dX_t=2\left(\frac{u_1'(\sqrt{X_t})\sqrt{X_t}}{u_1(\sqrt{X_t})}+1\right)\,dt
+2 \sqrt{X_t}\,d\beta_t, \qquad t \leq R_y,
\ee
for some $P^{\delta, u_1}_{z,y}$-Brownian motion $\beta$.
\item[ii)] For all $r\geq 0$
\[
Q^{\delta}_z\left[\exp\left(-r R_y -\frac{\lambda}{2}
    \Sigma^{\delta}_{p,z,y}\right)\right]=\frac{u_1(\sqrt{z})}{u_1(\sqrt{y})}\left(\frac{z}{y}\right)^{-\frac{\nu}{2}}
\frac{\Psi(z)}{\Psi(y)},
\]
where $\Psi$ is a positive and increasing solution of
\be \label{e:Psi}
2x v'' + 2
\left(\frac{u_1'(\sqrt{x})\sqrt{x}}{u_1(\sqrt{x})}+1\right)v'=rv
\ee
on $(0,\infty)$.
\item[iii)] A scale function of the diffusion defined in (\ref{e:Xamc2}) is given by
\be \label{e:stilde1}
\tilde{s_1}(x)= \int_1^x\frac{1}{y u_1^2(\sqrt{y})}dy, \qquad x\geq 0.
\ee
Thus, for any $0\leq a < z$
\[
Q^{\delta}_z\left[\chf_{[R_a >R_y]}\exp\left( -\frac{\lambda}{2}
    \Sigma^{\delta}_{p,z,y}\right)\right]= \frac{u_1(\sqrt{z})}{u_1(\sqrt{y})}\left(\frac{z}{y}\right)^{-\frac{\nu}{2}} \frac{\tilde{s_1}(z)-\tilde{s_1}(a)}{\tilde{s_1}(y)-\tilde{s_1}(a)}.
\]
\end{itemize}
\end{corollary}
\begin{proof} The proof follows the similar lines as in the proofs of analogous results for $y \leq z$. The only difference is that contrary to the previous case the drift term of the diffusion corresponding to the ODE in (\ref{e:Psi}) is only locally bounded. In particular, it is now larger than $2(\nu+1)$, which in turn implies that the solution is immediately pushed to $(0,\infty)$ as soon as it hits $0$ by comparison results for SDEs. Moreover, being only locally bounded causes no concern for our purposes since the computations involve the law of the diffusion until the first time it reaches $y$. 
\end{proof}

Again, since $I_{-\half}(x)=\sqrt{\frac{2 \pi}{x}}\cosh(x)$ we have
\begin{corollary} $p \geq 0$, $\nu \in (-1,0)$ and $\frac{\nu}{p+1}=-\half$. Then, for
  $y \geq z$ we have the following:
\begin{itemize}
\item[i)] \[
Q^{\delta}_z\left[\exp\left(-r R_y -\frac{\lambda}{2}
    \Sigma^{\delta}_{p,z,y}\right)\right]=\frac{\cosh\left(-\frac{\sqrt{\lambda}}{2 \nu}z^{-\nu}\right)}{\cosh\left(-\frac{\sqrt{\lambda}}{2 \nu}y^{-\nu}\right)}\frac{\Psi(z)}{\Psi(y)},
\]
where $\Psi$ is a positive and increasing solution of
\be \label{e:Psi0}
2x v'' + 2
\left(\nu +1 +\sqrt{\lambda} x^{-\nu}\tanh\left(-\frac{\sqrt{\lambda}}{2 \nu}x^{-\nu}\right)\right)v'=rv
\ee
on $(0,\infty)$.
\item[ii)] The function $\tilde{s_1}$ is, up to an affine transformation, given by $\tanh\left(-\sqrt{\lambda}\frac{x^{-\nu}}{2 \nu}\right)$.
\item[iii)] For $0 \leq a <z$,
\be
Q^{\delta}_z\left[\chf_{[R_a >R_y]}\exp\left( -\frac{\lambda}{2}
    \Sigma^{\delta}_{p,z,y}\right)\right]= \frac{\sinh\left(-\sqrt{\lambda}\frac{z^{-\nu}-a^{-\nu}}{
          2\nu}\right)}{\sinh\left(-\sqrt{\lambda}\frac{y^{-\nu}-a^{-\nu}}{
          2\nu}\right)}. \label{e:lm>a}
\ee
\end{itemize}
\end{corollary}
\begin{proof} Only part iii) needs proof. Note that
\[
Q^{\delta}_z\left[\chf_{[R_a >R_y]}\exp\left( -\frac{\lambda}{2}
    \Sigma^{\delta}_{p,z,y}\right)\right]=\frac{\cosh\left(-\frac{\sqrt{\lambda}}{2 \nu}z^{-\nu}\right)}{\cosh\left(-\frac{\sqrt{\lambda}}{2 \nu}y^{-\nu}\right)} \, \frac{\tanh\left(-\frac{\sqrt{\lambda}}{2 \nu}x^{-\nu}\right)-\tanh\left(-\frac{\sqrt{\lambda}}{2 \nu}a^{-\nu}\right)}{\cosh\left(-\frac{\sqrt{\lambda}}{2 \nu}y^{-\nu}\right)-\tanh\left(-\frac{\sqrt{\lambda}}{2 \nu}a^{-\nu}\right)}.
    \]
On the other hand,
\bean
&&\cosh\left(-\frac{\sqrt{\lambda}}{2 \nu}z^{-\nu}\right)\left\{\tanh\left(-\frac{\sqrt{\lambda}}{2 \nu}x^{-\nu}\right)-\tanh\left(-\frac{\sqrt{\lambda}}{2 \nu}a^{-\nu}\right)\right\}\\
&=&\frac{\sinh\left(-\frac{\sqrt{\lambda}}{2 \nu}z^{-\nu}\right)\cosh\left(-\frac{\sqrt{\lambda}}{2 \nu}a^{-\nu}\right)-\cosh\left(-\frac{\sqrt{\lambda}}{2 \nu}z^{-\nu}\right)\sinh\left(-\frac{\sqrt{\lambda}}{2 \nu}a^{-\nu}\right)}{\cosh\left(-\frac{\sqrt{\lambda}}{2 \nu}a^{-\nu}\right)}\\
&=&\frac{\sinh\left(-\frac{\sqrt{\lambda}}{2 \nu}z^{-\nu}\right)\cosh\left(\frac{\sqrt{\lambda}}{2 \nu}a^{-\nu}\right)+\cosh\left(-\frac{\sqrt{\lambda}}{2 \nu}z^{-\nu}\right)\sinh\left(\frac{\sqrt{\lambda}}{2 \nu}a^{-\nu}\right)}{\cosh\left(-\frac{\sqrt{\lambda}}{2 \nu}a^{-\nu}\right)}\\
&=&\frac{\sinh\left(\frac{\sqrt{\lambda}}{2 \nu}\left(a^{-\nu}-z^{-\nu}\right)\right)}{\cosh\left(-\frac{\sqrt{\lambda}}{2 \nu}a^{-\nu}\right)},
\eean
which yields the claimed representation.
\end{proof}

Note that in fact we do not need to assume $p>0$ for part iii) of the above
result to hold, since $X$ is never $0$ before $R_a$ for $0<a<z$. Moreover,  (\ref{e:lm>a}) and (\ref{e:lM<a}) are the same. Thus,
\begin{corollary} \label{c:sigcl-up} Suppose that $\frac{\nu}{p+1}=-\frac{1}{2}$, $\nu \in (-1,0)$ and $p>-1$. Then, for
  $y \geq z >a$ we have that the law of $\Sigma^{\delta}_{p,z,y}$
  conditioned on the event $[R_a>R_y]$ is that of  the first hitting
  time of $(a^{-\nu}-y^{-\nu})^2/4\nu^2$ by a $3$-dimensional squared
  Bessel process started at $(a^{-\nu}-z^{-\nu})^2/4\nu^2$.
\end{corollary}
 We end this section with a  scaling property which will be useful in
the subsequent section. It is a direct consequence of the scaling
property of $BESQ^{\delta}$ applied to the definition of $\Sigma^{\delta}_{p,0,y}$.
\begin{proposition} \label{p:scalingn} Suppose that $p\geq 0, \nu \in (-1,0)$. Then, we have the
  following identity in law for any $y\geq 0$:
\[
y^{p+1}\Sigma^{\delta}_{p,0,1}\eid \Sigma^{\delta}_{p,0,y}.
\]
\end{proposition}

\subsection{The case $\nu \geq 0$}
Recall that when $\nu \geq 0$ the point $0$ is polar for $X$. Thus,
one can prove without any difficulty  that $M^{(u_1)}$, where  $u_1$ is as  defined in Theorem \ref{t:numinus2}, is a martingale stopped at $R_y$. Note that for $\nu \geq 0$ the random variable $\Sigma^{\delta}_{p,0,y}$ is well-defined and finite even in the case $p \in (-1,0)$. Indeed, Corollary XI.1.12 in \cite{ry} show that 
\[
\int_0^1 X_s^p ds \eid \left(\int_0^1 Y_s^{-\frac{2p}{2+p}}ds\right)^{\frac{2+p}{p}},
\]
where $X$ is a $BESQ^{2(\nu+1)}(0)$  and $Y$ is a  $BESQ^{2\left(\nu\frac{2}{2+p}+1\right)}(0)$. Since $-\frac{2p}{2+p} >0$ whenever $p \in (-1,0)$ we easily deduce that the random variable on the left-hand side of the above identity is finite. Since a Bessel process with a positive dimension never comes back to $0$ again this yields the finiteness of $\int_0^{\eps} X_s^p ds$ for all $\eps>0$, which in turn implies the finiteness of $\Sigma^{\delta}_{p,0,y}$.
\begin{theorem} \label{t:m2} Suppose that $p>-1, y \geq z,$ and $\nu \geq  0$. Let
  $u_1$ be the function defined in Theorem \ref{t:numinus2}. Then,
  $(M^{(u_1)}_{t \wedge R_y})_{t \geq 0}$ is a bounded martingale.
\end{theorem}

Recall that $BESQ^{\delta}$ is transient when $\nu >0$, thus  $Q^{\delta}_z(R_y<\infty)=1$ whenever $y \geq z$. Consequently, we can deduce the following.
\begin{corollary} \label{c:LTxy+}Let $u_1$ be the function defined in Theorem
  \ref{t:numinus2} and suppose that $p>-1, \nu \geq 0$. Then, we have the following for all $z\leq y$:
  \begin{itemize}
  \item[i)] If  $r \geq 0$ and $U$ is  $\cF_{R_y}$-measurable,
\[
Q^{\delta}_z\left[\exp\left(-r U -\frac{\lambda}{2}
    \Sigma^{\delta}_{p,z,y}\right)\right]=\frac{u_1(\sqrt{z})}{u_1(\sqrt{y})}\left(\frac{z}{y}\right)^{-\frac{\nu}{2}}
P^{\delta, u_1}_{z,y}\left[\exp\left(-r U\right)\right],
\]
where $P^{\delta, u_1}_{z,y}$ is defined by $\frac{dP^{\delta, u_1
  }_{z,y}}{dQ^{\delta}_z}=M^{(u_1)}_{R_y}$. Moreover, under $P^{\delta,
  u_1}_{z,y}$, $X$ satisfies (\ref{e:Xamc2}).
\item[ii)] For all $r\geq 0$
\[
Q^{\delta}_z\left[\exp\left(-r R_y -\frac{\lambda}{2}
    \Sigma^{\delta}_{p,z,y}\right)\right]=\frac{u_1(\sqrt{z})}{u_1(\sqrt{y})}\left(\frac{z}{y}\right)^{-\frac{\nu}{2}}
\frac{\Psi(z)}{\Psi(y)},
\]
where $\Psi$ is a  positive and increasing solution of (\ref{e:Psi}) on $(0,\infty)$.
\item[iii)] For any $a < z$
\[
Q^{\delta}_z\left[\chf_{[R_a >R_y]}\exp\left( -\frac{\lambda}{2}
    \Sigma^{\delta}_{p,z,y}\right)\right]= \frac{u_1(\sqrt{z})}{u_1(\sqrt{y})}\left(\frac{z}{y}\right)^{-\frac{\nu}{2}} \frac{\tilde{s_1}(z§		§)-\tilde{s_1}(a)}{\tilde{s_1}(y)-\tilde{s_1}(a)},
\]
where $\tilde{s_1}$ is as defined in (\ref{e:stilde1}).
\end{itemize}
\end{corollary}

Analogous to Proposition \ref{p:scalingn} we  have the following
scaling property.
\begin{proposition} \label{p:scalingp} Suppose that $p> -1, \nu \geq 0$. Then, we have the
  following identity in law for any $y\geq 0$:
\[
y^{p+1}\Sigma^{\delta}_{p,0,1}\eid \Sigma^{\delta}_{p,0,y}.
\]
\end{proposition}
We now return to the case $y \leq z$.
\begin{proposition} \label{p:jmS} Let $u_0$ be the function defined in Theorem
  \ref{t:numinus} and suppose that $\nu \geq 0, p>-1$. Then, we have the following for all $z\geq y$:
  \begin{itemize}
  \item[i)] If  $r \geq 0$ and $U$ is  $\cF_{R_y}$-measurable,
  \[
Q^{\delta}_z\left[\exp\left(-r U -\frac{\lambda}{2}
    \Sigma^{\delta}_{p,z,y}\right)\right]=\frac{u_0(\sqrt{z})}{u_0(\sqrt{y})}\left(\frac{z}{y}\right)^{-\frac{\nu}{2}}
P^{\delta, u_0}_{z,y}\left[\exp\left(-r U\right)\right],
\]
where $P^{\delta, u_0}_{z,y}$ is defined by $\frac{d P^{\delta, u_0
  }_{z,y}}{dQ^{\delta}_z}=M^{(u_0)}_{R_y}$.  Moreover, under $P^{\delta,
  u_0}_{z,y}$, $X$ satisfies (\ref{e:Xamc}) until $R_y$.
\item[ii)] For all $r\geq 0$
\[
Q^{\delta}_z\left[\exp\left(-r R_y -\frac{\lambda}{2}
    \Sigma^{\delta}_{p,z,y}\right)\right]=\frac{u_0(\sqrt{z})}{u_0(\sqrt{y})}\left(\frac{z}{y}\right)^{-\frac{\nu}{2}}
\frac{\Phi(z)}{\Phi(y)},
\]
where $\Phi$ is a positive and decreasing solution of (\ref{e:Phi}) on $(0,\infty)$.
\item[iii)] For any $a > z$
\[
Q^{\delta}_z\left[\chf_{[R_a >R_y]}\exp\left( -\frac{\lambda}{2}
    \Sigma^{\delta}_{p,z,y}\right)\right]= \frac{u_0(\sqrt{z})}{u_0(\sqrt{y})}\left(\frac{z}{y}\right)^{-\frac{\nu}{2}} \frac{\tilde{s_0}(z)-\tilde{s_0}(a)}{\tilde{s_0}(y)-\tilde{s_0}(a)},
\]
where $\tilde{s_0}$ is as defined in (\ref{e:stilde0}).
\end{itemize}
\end{proposition}
\begin{proof}
We will only give the details for the parts of the proof that differ from the analogous results of the previous section. Observe that  $(M^{(u_0)}_{t \wedge R_y})_{t \geq 0}$
is a uniformly integrable martingale when $X$ starts at a value larger than $y$. Thus,
using the Optional Stopping Theorem, we obtain
\bean
u_0(\sqrt{z}) z^{-\frac{\nu}{2}} P^{\delta, u_0}_{z,y}\left[\exp\left(-r U\right)\right]&=&Q^{\delta}_z\left[\exp\left(-r U\right)M^{(u_0)}_{R_y}\right]\\
&=&Q^{\delta}_z\left[\chf_{[R_y
    <\infty]}u_0(\sqrt{y})y^{-\frac{\nu}{2}}\exp\left(-r U -\frac{\lambda}{2}\Sigma^{\delta}_{p,z,y}\right)\right]\\
    &&+Q^{\delta}_z\left[\chf_{[R_y
    =\infty]}M^{(u_0)}_{\infty}\right].
\eean
For $\nu=0$, $R_y$ is finite a.s., hence the claim. In case of
$\nu>0$, we still obtain the formula  since, on the set $[R_y=\infty]$, $\Sigma^{\delta}_{p,z,y}=\infty$ as well as   $u_0(\infty)=0$ due to the transience of $X$.

Using the recurrence relations for $K_{\alpha}$ as in Corollary \ref{c:ode-sdeK}, we can deduce that 
\[
\frac{u_0'(\sqrt{x})\sqrt{x}}{u_0(\sqrt{x})}\leq -\nu,
\]
and the left-hand side of the above inequality converging to the right-hand side as $x$ tends to $0$. Thus, solution  of the SDE given by  (\ref{e:Xamc})  until the first time it hits $0$ is less than that of 
\[
dX_t = 2 (-\nu+1) dt + 2 \sqrt{|X_t|}d\beta_t.
\]
Note that if $\nu >1$ the drift of the above SDE is negative hence its solution corresponds to a squared Bessel process of negative dimension. This SDE has a unique strong solution and it stays $(-\infty, 0]$ as soon as it hits $0$ (see Section 3 of \cite{gy}). Thus, we conclude by means of comparison results for SDEs that there exists a solution to (\ref{e:Xamc}) until $R_y$ for every $y>0$. 

The ODE (\ref{e:Phi}) has increasing and decreasing solutions which can be determined as before via the Laplace transforms of the first hitting times of the the diffusion on $(0, \infty)$ with the generator 
\[
\cA = 2x  \frac{d^2}{dx^2} + 2\left(\frac{u_0'(\sqrt{x})\sqrt{x}}{u_0(\sqrt{x})} +1\right)\frac{d}{dx},
\]
which is killed as soon as it reaches $0$. 

\end{proof}

As before, using the explicit form of $K_{\frac{1}{2}}$, one gets 
\begin{corollary} \label{c:bht+} Suppose that $\frac{\nu}{p+1}=\frac{1}{2}$ and $p>-1$. Then, for
  $y \leq z$ we have the following:
\begin{itemize}
\item[i)] \[
Q^{\delta}_z\left[\exp\left(-r R_y -\frac{\lambda}{2}
    \Sigma^{\delta}_{p,z,y}\right)\right]=\frac{y^{\nu}}{z^{\nu}}\exp\left(-\sqrt{\lambda}\frac{z^{\nu}-y^{\nu}}{2 \nu}\right)
\frac{\Phi(z)}{\Phi(y)},
\]
where $\Phi$ is a positive and decreasing solution of
\be \label{e:Phi00}
2x v'' + 2
\left(-\nu +1 -\sqrt{\lambda} x^{\nu}\right)v'=rv
\ee
on $(0,\infty)$.
\item[ii)] The function $\tilde{s_0}$ is, up to an affine transformation, given by $\exp\left(\sqrt{\lambda}\frac{x^{\nu}}{\nu}\right)$.
\item[iii)] For $a >z$
\bea
Q^{\delta}_z\left[\chf_{[R_a >R_y]}\exp\left( -\frac{\lambda}{2}
    \Sigma^{\delta}_{p,z,y}\right)\right]&=& \frac{y^{\nu}}{z^{\nu}}\exp\left(-\sqrt{\lambda}\frac{z^{\nu}-y^{\nu}}{2 \nu}\right)\frac{\exp\left(\sqrt{\lambda}\frac{z^{\nu}}{ \nu}\right)-\exp\left(\sqrt{\lambda}\frac{a^{\nu}}{ \nu}\right)}{\exp\left(\sqrt{\lambda}\frac{y^{\nu}}{ \nu}\right)-\exp\left(\sqrt{\lambda}\frac{a^{\nu}}{ \nu}\right)}\nn
    \\
    &=& \frac{y^{\nu}}{z^{\nu}}\frac{\sinh\left(\sqrt{\lambda}\frac{a^{\nu}-z^{\nu}}{ 2\nu}\right)}{\sinh\left(\sqrt{\lambda}\frac{a^{\nu}-y^{\nu}}{ 2\nu}\right)}.
\eea
\end{itemize}
\end{corollary}
As in the case with $\frac{\nu}{p+1}=-\frac{1}{2}$ we get
\bean
Q^{\delta}_z\left[\exp\left( -\frac{\lambda}{2}
    \Sigma^{\delta}_{p,z,y}\right)\bigg| R_a >R_y\right]&=&\frac{y^{\nu}}{z^{\nu}}\frac{y^{-\nu}-a^{-\nu} }{z^{-\nu}-a^{-\nu} }\frac{\sinh\left(\sqrt{\lambda}\frac{a^{\nu}-z^{\nu}}{
          2\nu}\right)}{\sinh\left(\sqrt{\lambda}\frac{a^{\nu}-y^{\nu}}{
          2\nu}\right)}\\
          &=&\frac{1-\left(\frac{y}{a}\right)^{\nu}}{1-\left(\frac{z}{a}\right)^{\nu}}\frac{\sinh\left(\sqrt{\lambda}\frac{a^{\nu}-z^{\nu}}{
          2\nu}\right)}{\sinh\left(\sqrt{\lambda}\frac{a^{\nu}-y^{\nu}}{
          2\nu}\right)}\\
      &=&\frac{a^{\nu}-y^{\nu}}{a^{\nu}-z^{\nu}}\frac{\sinh\left(\sqrt{\lambda}\frac{a^{\nu}-z^{\nu}}{
          2\nu}\right)}{\sinh\left(\sqrt{\lambda}\frac{a^{\nu}-y^{\nu}}{
          2\nu}\right)},
\eean
and hence
\begin{corollary}\label{c:sigcl+} Suppose that $\frac{\nu}{p+1}=\frac{1}{2}$ and $p>-1$. Then, for
  $y \leq z <a$ we have that the law of $\Sigma^{\delta}_{p,z,y}$
  conditioned on the event $[R_a>R_y]$ is that of  the first hitting
  time of $(a^{\nu}-y^{\nu})^2/4\nu^2$ by a $3$-dimensional squared
  Bessel process started at $(a^{\nu}-z^{\nu})^2/4\nu^2$.
\end{corollary}
Similarly, since $I_{\half}(x)=\sqrt{\frac{2 \pi}{x}}\sinh(x)$ we have
\begin{corollary} $p >-1$ and $\frac{\nu}{p+1}=\half$. Then, for
  $y \geq z$ we have the following:
\begin{itemize}
\item[i)] \[
Q^{\delta}_z\left[\exp\left(-r R_y -\frac{\lambda}{2}
    \Sigma^{\delta}_{p,z,y}\right)\right]=\frac{y^{\nu}\sinh\left(\frac{\sqrt{\lambda}}{2 \nu}z^{\nu}\right)}{z^{\nu}\sinh\left(\frac{\sqrt{\lambda}}{2 \nu}y^{\nu}\right)}\frac{\Psi(z)}{\Psi(y)},
\]
where $\Psi$ is a positive and increasing solution of
\be \label{e:Psi00}
2x v'' + 2
\left(-\nu +1 +\sqrt{\lambda} x^{\nu}\coth\left(\frac{\sqrt{\lambda}}{2 \nu}x^{\nu}\right)\right)v'=rv
\ee
on $(0,\infty)$.
\item[ii)] The function $\tilde{s_1}$ is, up to an affine transformation, given by $\coth\left(\sqrt{\lambda}\frac{x^{\nu}}{2 \nu}\right)$.
\item[iii)] For $0 \leq a <z$,
\be
Q^{\delta}_z\left[\chf_{[R_a >R_y]}\exp\left( -\frac{\lambda}{2}
    \Sigma^{\delta}_{p,z,y}\right)\right]= \frac{y^{\nu}}{z^{\nu}}\frac{\sinh\left(\sqrt{\lambda}\frac{z^{\nu}-a^{\nu}}{
          2\nu}\right)}{\sinh\left(\sqrt{\lambda}\frac{y^{\nu}-a^{\nu}}{
          2\nu}\right)}. \label{e:lm>a0}
\ee
\end{itemize}
\end{corollary}
\begin{remark} Comparing parts i) of Corollary \ref{c:bht-} and
  \ref{c:bht+} immediately gives us that, for $z \geq y$, the
  distributions of $\Sigma^{\delta}_{p,z,y}$ are different  when $\nu$
  has different signs. On the other hand, Corollaries \ref{c:sigcl-}
  and \ref{c:sigcl+} imply that they have the same distribution once
  they are conditioned on the event that the maximum of the underlying
  squared Bessel process is less than $a$ by time $R_y$. Same
  conclusion holds when $z\leq y$.
\end{remark}
\section{Small ball problem and Chung's law of iterated logarithm}
The small ball problem (also called small deviations) for a stochastic process $Z=(Z_t)_{t \in
  \mathcal{T}}$ consists in finding the probability
\[
\bbP[\|Z\| < \eps] \qquad \mbox{as } \eps \rar 0,
\]
$\| \cdot\|$ is a given norm, usually $L^p$ or $L^{\infty}$. It is
connected to many other questions, such as the law of the iterated
logarithm of Chung's type (Chung's LIL for short), strong limit laws in statistics, metric
entropy properties of linear operators and several  approximation
quantities for stochastic processes. The determination of the
above probability is not feasible other than in  a very few cases and
one is inclined to consider the asymptotic behaviour of
\[
-\log \bbP[\|Z\| < \eps] \qquad \mbox{as } \eps \rar 0.
\]
The solution to the latter problem is also not available in full
generality. However, one can get this asymptotic behaviour for
Gaussian processes (see, e.g.,  \cite{ls01} and \cite{lspa}) or real-valued L\'evy processes (see
\cite{ad}). There is a large amount of literature on small ball
probabilities in the Gaussian setting
and one can consult the survey article \cite{ls01}.

As one can expect from the computations made in the previous section, we will be interested in the small ball probabilities for the
stochastic process $(X_t)_{t \geq 0}$, and the
``norm''
\[
\|Z\|_{p,y}=\left(\int_0^{R_y} |Z_t|^p\,dt\right)^{\frac{1}{p}},
\]
where $p \in (0,\infty)$ and $R_y=\inf\{t>0:X_t=y\}$. Observe that the above
definition is not a real norm unless $p\geq 1$, however, as the
results  in this section does not depend on whether $\|\cdot\|_{p,y}$
is a true norm, this is not a problem. Our results and proofs are close in nature to the results of \cite{ksi}.

Interestingly, the small ball probabilities for $X$ under the above norm does not depend on its index, $\nu$, as seen from the next theorem.
\begin{theorem} \label{t:sbp} Let $X$ be a $BESQ^{\delta}$ as defined by (\ref{e:sdeX}) with $\delta>0$, and $R_y=\inf\{t>0:X_t=y\}$. Then, one has, for $z \geq 0$ and $y\geq 0$,
\bean
\lim_{\lambda \rar \infty} \lambda^{-\frac{1}{2}}\log Q_z^{\delta}\left[\exp\left(-\lambda \|X\|_{p,y}^p\right)\right]&=&-\frac{\sqrt{2}}{p+1}\left|z^{\frac{p+1}{2}}-y^{\frac{p+1}{2}}\right|\\
\lim_{\eps \rar 0} \eps^p \log  Q_z^{\delta}\left[\|X\|_{p,y}< \eps\right]
&=&-\frac{1}{2(p+1)^2}\left(z^{\frac{p+1}{2}}-y^{\frac{p+1}{2}}
\right)^2.
\eean
\end{theorem}
\begin{proof} Let $w(x)=u(\sqrt{x})x^{-\frac{\nu}{2}}$ where $u=u_0$ (resp. $u=u_1$) for $y \leq z$ (resp. $y >z$) and $u_0$ and $u_1$ are as defined in Theorems \ref{t:numinus} and \ref{t:numinus2}, respectively. Then, it follows from the results of the previous section that
\[
\sqrt{2}\lambda^{-\frac{1}{2}} \log Q_z^{\delta}\left[\exp\left(-\frac{\lambda}{2}\|X\|^p_{p,y}\right)\right]=\sqrt{2}\frac{\log w(z)}{\sqrt{\lambda}}-\sqrt{2}\frac{\log w(y)}{\sqrt{\lambda}}.
\]
Moreover, when $u=u_0$,
\bean
\lim_{\lambda \rar \infty} \frac{\log w(x)}{\sqrt{\lambda}}&=&\lim_{\lambda \rar \infty} \frac{\log u_0(\sqrt{x})}{\sqrt{\lambda}}\\
&=&\frac{x^{\frac{p+1}{2}}}{p+1}\lim_{\lambda \rar \infty} \frac{\log K_{\frac{|\nu|}{p+1}}\left(\frac{\sqrt{\lambda}}{p+1}x^{\frac{p+1}{2}}\right)}{\sqrt{\lambda}}\frac{p+1}{x^{\frac{p+1}{2}}}\\
&=&\frac{x^{\frac{p+1}{2}}}{p+1}\lim_{\lambda \rar \infty} \frac{\log K_{\frac{|\nu|}{p+1}}\left(\lambda\right)}{\lambda}.
\eean
However, using the asymptotic expansions in  (\ref{e:MBasymp})  we obtain that for any $\alpha \geq 0$,
\[
\lim_{\lambda \rar \infty} \frac{\log K_{\alpha}(\lambda)}{\lambda}=-1.
\] 
This shows that when $y \leq z$
\[
\lim_{\lambda \rar \infty} \lambda^{-\frac{1}{2}}\log Q_z^{\delta}\left[\exp\left(-\lambda \|X\|^p_{p,y}\right)\right]=-\frac{\sqrt{2}}{p+1}\left|z^{\frac{p+1}{2}}-y^{\frac{p+1}{2}}\right|.
\]
In order to show the above limit when $y>x$, it suffices to show that
$\lim_{x \rar \infty} \frac{\log
  I_{\frac{\nu}{p+1}}\left(x\right)}{x}=1$, which again follows from (\ref{e:MBasymp}). This completes the proof of the first assertion of
the theorem.

The second assertion follows by applying de Bruijn's exponential Tauberian theorem (see Theorem 4.12.9 in \cite{RegVar}) to $\alpha =-1$ and $\beta=  \frac{1}{2(p+1)^2}\left(z^{\frac{p+1}{2}}-y^{\frac{p+1}{2}}\right)^2.$
\end{proof}

Observe that $\Sigma^{\delta}_{p,0,y}$ is an increasing process when indexed by $y$. We will next use the above theorem to obtain Chung's LIL for
$(\Sigma^{\delta}_{p,0,y})_{y \geq 0}$.
\begin{theorem} \label{t:clil} Let $\phi(y):=\frac{y^{p+1}}{\log \log y}$. Then, for any $\nu > -1$ and $p>0$ one has
\[
\liminf_{y \rar \infty}
\frac{\Sigma^{\delta}_{p,0,y}}{\phi(y)}=\frac{1}{2(p+1)^2}, \,  Q^{\delta}_0\mbox{-a.s.}.
\]
\end{theorem}
\begin{proof} It follows form Theorem \ref{t:sbp} that
\[
\lim_{\eps\rar 0}\eps \log
Q^{\delta}_0\left[\Sigma^{\delta}_{p,0,1}<\eps\right]=-\frac{1}{2(p+1)^2},
\]
thus, for sufficiently small $\eps$,
\[
Q^{\delta}_0\left[\Sigma^{\delta}_{p,0,1}<\eps\right]\leq
\exp\left(-\frac{K}{\eps}\right),
\]
where $K$ is a fixed, but arbitrary, constant in $(0,
\frac{1}{2(p+1)^2})$. Fix $C>1$ and set $y_n=C^n$. Next choose $k>0$
so that $k C^{p+1} <K$,  In view of
Propositions \ref{p:scalingn} and \ref{p:scalingp}, we get for all
large $n$
\bean
Q^{\delta}_0\left[\Sigma^{\delta}_{p,0,y_{n}}<k \phi(y_{n+1})\right]&=&
Q^{\delta}_0\left[\Sigma^{\delta}_{p,0,1}<\frac{k C^{p+1}}{\log \log y_{n+1}}
\right]\\
&\leq& \exp\left(-\frac{K}{k C^{p+1}}\log \left((n+1) \log C\right)\right)=(\log
C)^{-\frac{K}{k C^{p+1}}} (n+1)^{-\frac{K}{k C^{p+1}}},
\eean
which is summable in $n$. Therefore, by the first Borel-Cantelli
lemma, we have that, a.s. for large $n$,
$\frac{\Sigma^{\delta}_{p,0,y_n}}{\phi(y_{n+1})} \geq k$. On the
other hand, for $y \in [y_n, y_{n+1}]$,
\[
\Sigma^{\delta}_{p,0,y} \geq \Sigma^{\delta}_{p,0,y_n}\geq k
\phi(y_{n+1})\geq k \phi(y),
\]
which shows that
\[
\liminf_{y \rar \infty}
\frac{\Sigma^{\delta}_{p,0,y}}{\phi(y)}\geq k, \,  Q^{\delta}_0\mbox{-a.s.},
\]
and thus
\[
\liminf_{y \rar \infty}
\frac{\Sigma^{\delta}_{p,0,y}}{\phi(y)}\geq \frac{1}{2(p+1)^2}, \,  Q^{\delta}_0\mbox{-a.s.},
\]
by the arbitrariness of $C, K$ and $k$.

We now turn to prove the reverse inequality.  First, let's observe
that \be \label{e:finli}
\liminf_{y \rar \infty} \frac{\Sigma^{\delta}_{p,0,y}}{y^{p+1}} < \infty,\,  Q^{\delta}_0\mbox{-a.s.}.
\ee
The above claim follows from a direct
  application of Fatou's lemma since $\Sigma^{\delta}_{p,0,y}\eid
  y^{p+1}\Sigma^{\delta}_{p,0,1}$.

 Next, fix an $\eps>0$, let $y_n=n^n$ and consider the
  events
\[
E_n:=\left[ \int_{R_{y_{n-1}}}^{R_{y_n}}X_s^p\,ds \leq (1+ 2
  \eps)\frac{1}{2(p+1)^2}\phi(y_n)\right].
\]
It follows from the strong Markov property of $X$ that $E_n$s are
independent, and we will now see that $E_n$s occur infinitely often
due to the second Borel-Cantelli lemma. Indeed, by the definition
of $\Sigma^{\delta}_{p,0,y_n} $, we obtain
\bean
Q^{\delta}_0(E_n)& \geq&
Q^{\delta}_0\left[\Sigma^{\delta}_{p,0,y_n}\leq (1+ 2
  \eps)\frac{1}{2(p+1)^2}\phi(y_n)\right]\\
&=& Q^{\delta}_0\left[\Sigma^{\delta}_{p,0,1}\leq (1+ 2
  \eps)\frac{1}{2(p+1)^2 \log\log y_n}\right]\\
&\geq&\exp\left(-\frac{1+\eps}{1+2\eps}\log\log y_n\right)\geq
\frac{1}{\log y_n},
\eean
where the second inequality is due to the fact that, for a given
$\eps>0$, $Q^{\delta}_0 \left[\Sigma^{\delta}_{p,0,1}\leq
  \eta\right]\geq \exp\left(-(1+\eps)\frac{1}{2(p+1)^2
  }\frac{1}{\eta}\right)$ for sufficiently small
$\eta$ in view of the convergence result of Theorem
\ref{t:sbp}. Since $\frac{1}{n \log n}$ is not summable, it follows
from the Borel-Cantelli lemma that $E_n$ occurs infinitely often. As
$\eps$ was arbitrary this allows us to conclude, a.s.,
\[
\liminf_{n \rar \infty}
\frac{\int_{R_{y_{n-1}}}^{R_{y_n}}X_s^p\,ds}{\phi(y_n)}\leq
\frac{1}{2(p+1)^2}.
\]
Thus, $Q^{\delta}_0$-a.s.,
\be  \label{e:finli2}
\liminf_{n \rar \infty}
\frac{\Sigma^{\delta}_{p,0,y_n}}{\phi(y_n)}\leq \liminf_{n \rar \infty}
\frac{\Sigma^{\delta}_{p,0,y_{n-1}}}{\phi(y_n)} +\frac{1}{2(p+1)^2}=\liminf_{n \rar \infty}
\frac{\Sigma^{\delta}_{p,0,y_{n-1}}}{y^{p+1}_{n-1}}\frac{y^{p+1}_{n-1}}{\phi(y_n)} +\frac{1}{2(p+1)^2}.
\ee
On the other hand,
\[
\frac{y_{n-1}^{p+1}}{\phi(y_n)} \leq \frac{\log \log n^{n(p+1)}
  }{n^{p+1}},
\]
which converges to $0$ as $n\rar \infty$. Therefore, in view of
(\ref{e:finli}) and (\ref{e:finli2}), we obtain
\[
\liminf_{n \rar \infty}
\frac{\Sigma^{\delta}_{p,0,y_n}}{\phi(y_n)}\leq \frac{1}{2(p+1)^2}, \,  Q^{\delta}_0\mbox{-a.s.}.
\]
\end{proof}
\begin{remark} We can in fact extend the above result so that for all $z \geq 0$
\[
\liminf_{y \rar \infty}
\frac{\Sigma^{\delta}_{p,z,y}}{\phi(y)}=\frac{1}{2(p+1)^2}, \,  Q^{\delta}_z\mbox{-a.s.}.
\]
Indeed, using the strong Markov property of $X$ we have the decomposition
\[
\Sigma^{\delta}_{p,0,y}=\Sigma^{\delta}_{p,0,z}+ \Sigma^{\delta}_{p,z,y}
\]
where $\Sigma^{\delta}_{p,0,z}$ and $ \Sigma^{\delta}_{p,z,y}$ are independent. Dividing both sides by $\phi(y)$ and letting $y \rar \infty$ yields the result. 

 Another formal check to this result can be performed by observing 
\[
\Sigma^{\delta}_{p,z,y}\eid y^{p+1}\Sigma^{\delta}_{p,\frac{z}{y},1}\eid y^{p+1}\Sigma^{\delta}_{p,0,1}\frac{\Sigma^{\delta}_{p,\frac{z}{y},1}}{\Sigma^{\delta}_{p,0,1}}\eid\Sigma^{\delta}_{p,0,y}\frac{\Sigma^{\delta}_{p,\frac{z}{y},1}}{\Sigma^{\delta}_{p,0,1}},
\]
and that  $\frac{\Sigma^{\delta}_{p,\frac{z}{y},1}}{\Sigma^{\delta}_{p,0,1}}$ converges to $1$ as $y \rar \infty$ once we identify $\Sigma^{\delta}_{p,\frac{z}{y},1}$ with
\[
\int_{R_{\frac{z}{y}}}^{R_1} X_s^p \, ds,
\]
where $X$ is $BESQ^{\delta}(0)$. 
\end{remark}
\section{Feller property and `time reversal'}
In the previous section we have proved a law of iterated logarithm for $\Sigma^{\delta}_{p,0,y}$ by considering it as a process indexed by $y$. In this section we will  see, for $\nu\geq 0$,  that it is in fact an {\em inhomogeneous} Feller process and find its infinitesimal generator.

First of all, it immediately follows from the strong Markov property of $X$ that $(\Sigma^{\delta}_{p,0,y}, \cF_{R_y})_{y \geq 0}$ is Markov. Suppose $P_{z,y}$ is the associated semigroup, i.e. $P_{z,y}f(a)=Q^{\delta}_0[f(\Sigma^{\delta}_{p,0,y})|\Sigma^{\delta}_{p,0,z}=a].$ Since the increments of $(\Sigma^{\delta}_{p,0,y})_{y \geq 0}$ are independent, we have for any bounded measurable $f$
\be \label{tfS}
P_{z,y}f(a)=\int_0^{\infty} f(a+b)Q^{\delta}_z(\Sigma^{\delta}_{p,z,y}\in db).
 \ee
Let $C_0$ denote the class of continuous functions on $\bbR_+$ that vanish at  $\infty$. (\ref{tfS}) readily implies that when $f \in C_0$, $P_{z,y}f \in C_0$ as well. Moreover,  it follows from Corollary \ref{c:LTxy+}, and the observation that $u_1$ is finite at $0$, that for each $z \geq 0$ the measure $Q^{\delta}_z(\Sigma^{\delta}_{p,z,y}\in db)$ converges weakly to the Dirac point mass at $0$ as $y \dar z$ since its Laplace transform converges to $1$. Therefore, $\lim_{y \dar z}P_{z,y}f(a)=f(a)$ and consequently $(\Sigma^{\delta}_{p,0,y}, \cF_{R_y})_{y \geq 0}$ is Feller.

The form of the infinitesimal generator of $(\Sigma^{\delta}_{p,0,y}, \cF_{R_y})_{y \geq 0}$ will follow from the following theorem.

\begin{theorem} \label{t:ig1} Suppose that $\nu \geq 0$. Then, for every $x \geq 0$ there exists a decreasing function $\ol{\pi}(x, \cdot)$  satisfying
\[
\int_0^{\infty}e^{-\frac{\lambda}{2} b}\ol{\pi}(x,b)db=\frac{2}{\lambda}\frac{w'(x)}{w(x)}, \qquad \lambda \in (0,\infty),
\]
where $w(x):=x^{-\frac{\nu}{2}}I_{\frac{\nu}{p+1}}\left(\frac{\sqrt{\lambda}}{p+1}x^{\frac{p+1}{2}}\right)$.
Moreover,
\be
\lim_{y \dar z} \frac{Q^{\delta}_0\left[\exp\left(-\lambda \Sigma^{\delta}_{p,0,y}\right)\bigg|\Sigma^{\delta}_{p,0,z}=a\right]-e^{-\lambda a}}{y-z}=\int_0^{\infty}\left\{e^{-\lambda (a+b)} -e^{-\lambda a}\right\}\pi(z,db),
\ee
where $\pi(z, db):=-\ol{\pi}(z, db)$ for $b \geq 0$. In particular,
\[
\int_0^1 b \pi(z,db)<\infty.
\]
\end{theorem}
\begin{proof}
In view of the strong Markov property of $X$,
\[
Q^{\delta}_0\left[\exp\left(-\frac{\lambda}{2} \Sigma^{\delta}_{p,0,y}\right)\bigg|\Sigma^{\delta}_{p,0,z}=a\right]=e^{-\frac{\lambda}{2} a}Q^{\delta}_z\left[\exp\left(-\frac{\lambda}{2}  \Sigma^{\delta}_{p,z,y}\right)\right].
\]

Thus, it follows from Corollary \ref{c:LTxy+} that
\bea
\lim_{y \dar z} \frac{Q^{\delta}_0\left[\exp\left(-\frac{\lambda}{2}  \Sigma^{\delta}_{p,0,y}\right)\bigg|\Sigma^{\delta}_{p,0,z}=a\right]-e^{-\frac{\lambda}{2}  a}}{y-z}&=&e^{-\frac{\lambda}{2}  a}\lim_{y \dar z}\frac{Q^{\delta}_z\left[\exp\left(-\frac{\lambda}{2}  \Sigma^{\delta}_{p,z,y}\right)\right]-1}{y-z}  \nn \\
&=&e^{-\frac{\lambda}{2}  a}\lim_{y \dar z} \frac{\frac{w(z)}{w(y)}-1}{y-z}=-e^{-\frac{\lambda}{2}a}\frac{w'(z)}{w(z)}. \label{e:frhs}
\eea
On the other hand, using integration by parts, we obtain
\bea
Q^{\delta}_z\left[\exp\left(-\frac{\lambda}{2} \Sigma^{\delta}_{p,z,y}\right)\right]-1&=&\int_0^{\infty}e^{-\frac{\lambda}{2} b}Q^{\delta}_z(\Sigma^{\delta}_{p,z,y} \in db)-1 \nn \\&=&-\frac{\lambda}{2}  \int_0^{\infty}e^{-\frac{\lambda}{2}  b}Q^{\delta}_z(\Sigma^{\delta}_{p,z,y} >b)\, db. \label{e:lap}
\eea
It is well-known (see Corollary 3.8 in Chap.~VII of \cite{ry}) that
\[
Q^{\delta}_z[\Sigma^{\delta}_{p,z,y}]=\int_0^y x^p G(z,x) m(dx),
\]
where $m$ and $G$ are the associated speed measure and Green's function, respectively. In our case, these are given by
\[
m(dx)=\frac{x^{\nu}}{2 \nu} dx; \qquad G(z,x)= -y^{-\nu}+ (x \vee z)^{-\nu}.
\]
Consequently, the above formula yields
\[
Q^{\delta}_z[\Sigma^{\delta}_{p,z,y}]=\frac{y^{p+1}-z^{p+1}}{2 (p+1)(p+\nu +1)},
\]
which in particular implies that $\Pi(z,y, db):= \frac{2 (p+1)(p+\nu
  +1)}{y^{p+1}-z^{p+1}}Q^{\delta}_z(\Sigma^{\delta}_{p,z,y} >b) db$ is
a probability measure on $[0,\infty)$ for each $(z,y)$. However,
(\ref{e:frhs}) and (\ref{e:lap}) imply that $\Pi(z,y, \cdot)$ converges
weakly as $y$ tends to $z$ to some probability measure, $\ol{\Pi}(z,\cdot)$
on $[0,\infty)$ which satisfies\footnote{Using the integral representation of $I_{\alpha}$ for $\alpha > -\half$, it is tedious but straightforward to  check that this representation holds for $z=0$ as well by taking the limit as $z \rar 0$ and showing that $L(\lambda) < \infty$ for $\lambda >0$.}
\be \label{e:PiL}
L\left(\frac{\lambda}{2}\right):=\int_0^{\infty}e^{-\frac{\lambda}{2} b}\ol{\Pi}(z,db)=\frac{4(p+\nu+1)}{\lambda}z^{-p}\frac{w'(z)}{w(z)}.
\ee
Moreover, since $Q^{\delta}_z(\Sigma^{\delta}_{p,z,y} >b)$ is
decreasing in $b$ for each $(z,y)$, the limiting measure $\ol{\Pi}$ is necessarily of the form
\[
c\eps_0(db) + 2(p+\nu+1)z^{-p}\overline{\pi}(z,b)\,db,
\]
where $\eps_0$ is the Dirac point mass at $0$, $c$ a nonnegative
constant, and $\overline{\pi}(z,\cdot)$ is a decreasing function for
each $x$. In particular, $\overline{\pi}(z,\infty)=0$.  In order to
find the constant $c$, it suffices to check the value of the function
$L$ at $\infty$.  However, using the explicit form of $w$,
\[
c=\lim_{\lambda \rar
  \infty}\frac{4(p+\nu+1)}{\lambda}z^{-p}\frac{w'(z)}{w(z)}=\lim_{\lambda
  \rar \infty}2 \frac{p+\nu+1}{\sqrt{\lambda} z^{\frac{p+1}{2}}}\frac{I'_{\frac{\nu}{p+1}}\left(\frac{\sqrt{\lambda}}{p+1}z^{\frac{p+1}{2}}\right)}{I_{\frac{\nu}{p+1}}\left(\frac{\sqrt{\lambda}}{p+1}z^{\frac{p+1}{2}}\right)}
=0
\]
since
\[
\lim_{x \rar \infty}\frac{1}{x}\frac{I'_{\alpha}(x)}{I_{\alpha}(x)}=2\lim_{x \rar \infty} \frac{\log I_{\alpha}(x)}{x^2}=0
\]
due to (\ref{e:MBasymp}). 

Thus, we have shown that
\[
-\frac{w'(z)}{w(z)} =\lim_{y \dar z}\frac{Q^{\delta}_z\left[\exp\left(-\frac{\lambda}{2}  \Sigma^{\delta}_{p,z,y}\right)\right]-1}{y-z}=-\frac{\lambda}{2} \int_0^{\infty}e^{-\frac{\lambda}{2}  b} \overline{\pi}(z, b)\,db.
\]
Since $\ol{\pi}$ is decreasing with $\ol{\pi}(z,\infty)=0$, we obtain by integrating by parts
\[
\lim_{y \dar z} \frac{Q^{\delta}_0\left[\exp\left(-\frac{\lambda}{2}  \Sigma^{\delta}_{p,0,y}\right)\bigg|\Sigma^{\delta}_{p,0,z}=a\right]-e^{-\frac{\lambda}{2}  a}}{y-z}=e^{-\frac{\lambda}{2}  a}\int_0^{\infty}\left\{e^{-\frac{\lambda}{2}  b}-1\right\}\pi(z,db)
\]
where $\pi(z,db)=-\ol{\pi}(z,db)$. Finally, note that one necessarily has
\[
\int_0^1 b \pi(z,db)<\infty,
\]
since otherwise $L$, as defined in  (\ref{e:PiL}), would have been infinite.
\end{proof}

 The above theorem yields  that the sequence of measures $\left(\frac{Q^{\delta}_z(\Sigma^{\delta}_{p,z,y} >b)}{y-z} db\right)$ converges vaguely to a finite measure on $(0,\infty)$ as $y \rar z$. Thus, for any $f \in C^1_K(\bbR_+, \bbR)$, i.e. the space of continuously differentiable functions with a compact support, we have
\be
\lim_{y \dar z} \frac{Q^{\delta}_0\left[f\left(\Sigma^{\delta}_{p,0,y}\right)|\Sigma^{\delta}_{p,0,z}=a\right]-f(a)}{y-z}= \int_0^{\infty}\left\{f(a+b) -f(a)\right\}\pi(z,db).
\ee
In other words, letting $B(\bbR_+)$ denote the bounded Borel functions defined on $\bbR_+$, if we define the operator ${\cA}_z:B(\bbR_+)\mapsto B(\bbR_+)$ by setting
\[
{\cA}_z f (a):= \int_0^{\infty}\left\{f(a+b) -f(a)\right\}\pi(z,db) \qquad \mbox{for } f \in  C^1_K(\bbR_+, \bbR),
\]
then we see that the process $(M^f_y)_{y \geq 0}$ defined by
\[
M^f_y:=f(\Sigma^{\delta}_{p,0,y})-\int_0^y {\cA}_z f(\Sigma^{\delta}_{p,0,z})\,dz
\]
is a martingale with respect to the filtration $(\cF_{R_y})_{y \geq 0}$ whenever $f$ belongs to the domain of ${\cA}_z$ for all $z \geq 0$.  The form of the infinitesimal generator also reveals the fact that the increasing process $(\Sigma^{\delta}_{p,0,y})_{y \geq 0}$ is purely discontinuous, i.e. there is no interval $(a,b)$ in which it is continuous.
\begin{remark}It follows from the fact that $(R_y)_{y \geq 0}$ is left continuous that $(\Sigma^{\delta}_{p,0,y})_{y\geq 0}$ is a left-continuous process. However, in view of the above Feller property one can obtain a \cadlag version of it when we augment the filtration with the null sets.  Existence of a right-continuous version can also be independently verified by observing that
\[
\lim_{y \dar z}Q^{\delta}_0\left[\exp\left(-\lambda \Sigma^{\delta}_{p,0,y}\right)\right]=Q^{\delta}_0\left[\exp\left(-\lambda \Sigma^{\delta}_{p,0,z}\right)\right].
\]
\end{remark}

We will end this section by analysing a specific `time reversal' example. To this end let $L_x:=\sup\{t\geq 0:X_t=x\}$ and suppose that $\nu >0$ so that $Q^{\delta}_0(L_x < \infty)=1$ for all $x\geq 0$. We will consider the process $Z^{\delta}$ defined by
\be \label{e:zdelta}
Z^{\delta}_x:=\int_{L_{1-x}}^{L_1}X_s^p\,ds \qquad \forall x \in[0,1).
\ee
 In view of the well-known time reversal results for diffusions, see, e.g., Exercise 1.23 in Chap.~XI of \cite{ry}, the law of the process $(X_{L_1-t}, t < L_1)$ under $Q^{\delta}_0$ is identical to that of $(X_t, t < {R}_0)$ under $Q^{2-2\nu}_1$. Recall that $Q^{2-2\nu}_1(R_0<\infty)=1$. Thus, we can write
  \be
  Z^{\delta}_x=\int_0^{{R}_{1-x}}X_s^p \,ds, \qquad X=BESQ^{2-2\nu}(1). \label{e:zdelta2}
  \ee
  Note that the above equality is to be understood in the sense of
  equality between the laws of the processes. Due to the strong Markov
  property of $X$ we again have that $(Z^{\delta}_x)_{x \in [0,1)}$ is
  a Markov process with respect to the filtration
  $(\overset{\lar}{\cF}_x)_{x \in [0,1)}$ where
  $\overset{\lar}{\cF}_x:=\sigma(X_s; L_{1-x}\leq s \leq L_1)$.  Observe,
  more easily in view of (\ref{e:zdelta2}), that $Z^{\delta}$, too, has independent increments  rendering its Feller property in view of the arguments that led to the Feller property of $(\Sigma^{\delta}_{p,0,z})_{z \geq 0}$ at the beginning of this section. The next theorem will yield the form of the infinitesimal generator. Its proof will follow similar lines of the proof of Theorem \ref{t:ig1} so we will only give the details when it differs.
\begin{theorem}\label{t:ig2} Let $Z^{\delta}$ be as defined in
  (\ref{e:zdelta}) and suppose $\nu \in (0,1]$. Then, for every $x \in [0,1)$ there exists a decreasing function $\tilde{\pi}(x, \cdot)$  satisfying
\[
\int_0^{\infty}e^{-\frac{\lambda}{2} b}\tilde{\pi}(x,b)db=-\frac{2}{\lambda}\frac{w'(1-x)}{w(1-x)}, \qquad \lambda \in (0,\infty),
\]
where $w(x):=x^{\frac{\nu}{2}}K_{\frac{\nu}{p+1}}\left(\frac{\sqrt{\lambda}}{p+1}x^{\frac{p+1}{2}}\right)$.
Moreover,
\be
\lim_{y \dar x} \frac{Q^{\delta}_0\left[\exp\left(-\lambda Z^{\delta}_y\right)\big|Z^{\delta}_x=a\right]-e^{-\lambda a}}{y-x}=\int_0^{\infty}\left\{e^{-\lambda (a+b)} -e^{-\lambda a}\right\}\rho(x,db),
\ee
where $\rho(x, db):=-\tilde{\pi}(x, db)$ for $b \geq 0$. In particular,
\[
\int_0^1 b \rho(x,db)<\infty.
\]
\end{theorem}
\begin{proof} First observe that in view of (\ref{e:zdelta2}) and the aforementioned independent increments property
\[
Q^{\delta}_0\left[\exp\left(-\frac{\lambda}{2} Z^{\delta}_y\right)\bigg|Z^{\delta}_x=a\right]=e^{-\frac{\lambda}{2}a}Q^{\beta}_{1-x}\left[\exp\left(-\frac{\lambda}{2}  \Sigma^{\beta}_{p,1-x,1-y}\right)\right],
\]
where $\beta=2-2\nu$. Therefore, the same arguments in the beginning of the proof of Theorem \ref{t:ig1} yields that the measures $\frac{Q^{\beta}_{1-x}(\Sigma^{\beta}_{p,1-x,1-y} >b) db}{y-x}$ converges vaguely
 to some probability measure, $\tilde{\Pi}(x,\cdot)$
on $[0,\infty)$ which satisfies
\be \label{e:PiL2}
L\left(\frac{\lambda}{2}\right):=\int_0^{\infty}e^{-\frac{\lambda}{2} b}\tilde{\Pi}(x,db)=-\frac{1}{\lambda}\frac{w'(1-x)}{w(1-x)}.
\ee
Moreover, the limiting measure is necessarily of the form
\[
c \eps_0(db) + \tilde{\pi}(x,b)\,db,
\]
where $\eps_0$ is the Dirac point mass at $0$, $c$ a nonnegative
constant, and $\tilde{\pi}(x,\cdot)$ is a decreasing function for
each $x$. Observe that since $Q^{\beta}_{1-x}(\Sigma^{\beta}_{p,1-x,1-y})=\infty$,  the measures $\frac{Q^{\beta}_{1-x}(\Sigma^{\beta}_{p,1-x,1-y} >b) db}{y-x}$ are not finite and neither is their vague limit. On the other hand, we can still conclude that $c=0$ since, in view of (\ref{e:MBasymp}),  one has
\[
\lim_{x \rar \infty}  \frac{1}{x}\frac{K'_{\alpha}(x)}{K_{\alpha}(x)}=2 \lim_{x \rar \infty}\frac{\log K_{\alpha}(x)}{x^2}=0.
\]
Next, integrating
(\ref{e:PiL2}) by parts yields
\be \label{e:exp1}
\tilde{\pi}(x,\infty)
-\int_0^{\infty}(1-e^{-\frac{\lambda}{2}b})\tilde{\pi}(x,db)=-\frac{w'(1-x)}{w(1-x)}.
\ee
Thus, due to the dominated convergence theorem,  taking the limit of
the above as $\lambda$ tends to $0$ yields
\[
\tilde{\pi}(x,\infty)=-\lim_{\lambda \rar 0}\frac{w'(1-x)}{w(1-x)}.
\]
Also note that using the asymptotic behaviour of $K$ near $0$  (\ref{e:MBasymp}) one has that
\[
\lim_{x\rar 0} x \frac{K'_{\alpha}(x)}{K_{\alpha}(x)}=\lim_{x\rar 0}\frac{\log K_{\alpha}(x)}{\log x}=-\alpha. 
\]
for any $\alpha\geq 0$.   This in
turn yields that
\[
\lim_{\lambda \rar
  0}\frac{w'(x)}{w(x)}=\frac{\nu}{2 x}+\frac{p+1}{2 x}\lim_{\lambda \rar
  0}\frac{K'_{\frac{\nu}{p+1}}\left(\frac{\sqrt{\lambda}}{p+1}x^{\frac{p+1}{2}}\right)}{K_{\frac{\nu}{p+1}}\left(\frac{\sqrt{\lambda}}{p+1}x^{\frac{p+1}{2}}\right)}\frac{\sqrt{\lambda}}{p+1}x^{\frac{p+1}{2}}=0,
\]
 and thus $\tilde{\pi}(x,\infty)=0$.
As in the proof of Theorem \ref{t:ig1}
\[
e^{-\frac{\lambda}{2}
    a}\frac{w'(1-x)}{w(1-x)}=\lim_{y \dar x}
\frac{Q^{\delta}_0\left[\exp\left(-\frac{\lambda}{2}
      Z^{\delta}_y\right)\big|Z^{\delta}_x=a\right]-e^{-\frac{\lambda}{2}
    a}}{y-x}.
\]
Thus, combining above with (\ref{e:exp1}) and plugging in the value of
$\tilde{\pi}(x,\infty)$ yield
\[
\lim_{y \dar x}
\frac{Q^{\delta}_0\left[\exp\left(-\frac{\lambda}{2}
      Z^{\delta}_y\right)\big|Z^{\delta}_x=a\right]-e^{-\lambda
    a}}{y-x}=\int_0^{\infty}\left\{e^{-\frac{\lambda}{2} (a+b)}
-e^{-\frac{\lambda}{2} a}\right\}\rho(x,db).
\]
\end{proof}
\begin{example} As an application of the above theorem consider the
  case when $\nu \in (0,1]$ and $\frac{\nu}{p+1}=\half$. Then, the associated
  $\tilde{\pi}$ is defined by
\[
\int_0^{\infty}e^{-\frac{\lambda}{2}
  b}\tilde{\pi}(x,b)db=\frac{1}{\sqrt{\lambda}}(1-x)^{\nu-1}
\]
in view of the explicit form for $K_{\frac{1}{2}}$. Thus, by inverting the above
transform, we have
\[
\tilde{\pi}(x,b)=(1-x)^{\nu-1}\frac{1}{\sqrt{2 \pi b}}.
\]
This reveals that the infinitesimal generator, $\tilde{A}_x$ of
$Z^{\delta}$ is defined by
\[
\tilde{A}_x=(1-x)^{\nu-1}\int_0^{\infty}\left\{f(a+b)-f(a)\right\} \frac{1}{2
    \sqrt{2 \pi b^3}}\,db
\]
for any $f$ in $C^1$ with a compact support. Consequently,
\[
f(Z_x^{\delta})-\int_0^x (1-y)^{\nu-1}\int_0^{\infty}\left\{f(Z^{\delta}_y+b)-f(Z^{\delta}_y)\right\} \frac{1}{2
    \sqrt{2 \pi b^3}}\,db\, dy
\]
is an $\overset{\lar}{\cF}$-martingale for such $f$.
\end{example}
\begin{example} Observe that although $Z^{\delta}$, or
  $(\Sigma^{\delta}_{p,0,y})$, is an increasing process with
  independent increments, it is not a {\em subordinator} (see \cite{bs} for a definition and further properties) since the
  increments are not stationary. However, in the above example, if one takes $\nu=1$, then one
  sees that $Z^{\delta}$ becomes time homogeneous, i.e. $Z^{\delta}$
  is a subordinator. Observe that $\nu=1$ implies $p=1$ in this
  framework. More precisely,
\[
Z^{\delta}_x=Z^{4}_x=\int_{L_{1-x}}^{L_1}X_s\,ds,
\]
where $X$ is $BESQ^4(1)$, is a subordinator. Moreover, Corollary
\ref{c:bht-} and (\ref{e:zdelta2}) yield
\[
Q^{\delta}_0\left(\exp\left(-\lambda
    Z^{4}_x\right)\right)=\exp\left(-\sqrt{\frac{\lambda}{2}}x\right).
\]
Thus,  $Z^{4}_x\eid T_{\frac{x}{2}}$, where $T_x$ is the first hitting
time of $x$ for a Brownian motion starting at $0$.
\end{example}
\section{Applications to finance}
Our aim in this section is to give some examples arising from some
financial models and discuss how the results from previous sections can be used to obtain prices of certain financial products.

As explained in Introduction the process  $X$ is commonly used in the finance literature to model interest rates. Suppose the spot interest rate is given by $X^p$ where $p>-1$ and $X$ is $BESQ^{\delta}(z)$ and consider the following {\em exotic} derivative on interest rates which pays one unit of a currency at time $R_y$ if the accumulated interest is less than $k$, i.e. $\Sigma^{\delta}_{p,z,y} \leq \log k$. This is an example of a digital option and its price, as usual in the Finance Theory, is given as an expectation of its discounted payoff:
\[
D(k;\delta,p,z,y)=Q^{\delta}_x\left[\chf_{[\Sigma^{\delta}_{p,z,y} \leq k]}\exp\left(-\Sigma^{\delta}_{p,z,y}\right)\right].
\]
On the other hand, if one computes the Laplace transform of $D(k;\delta,p,z,y)$, one obtains
\[
\int_0^{\infty}e^{-\mu k} D(k;\delta,p,z,y)\,dk=Q^{\delta}_z\left[\int_{\Sigma^{\delta}_{p,z,y}}^{\infty}e^{-\mu k}\exp\left(-\Sigma^{\delta}_{p,z,y}\right)\,dk\right]=Q^{\delta}_z\left[\exp\left(-(\mu+1)\Sigma^{\delta}_{p,z,y}\right)\right],
\]
which is at our disposal due to the results from Section \ref{s:2}. In
particular, the above identity implies
\[
\int_0^{\infty}e^{-\mu k} e^{k} D(k;\delta,p,z,y)\,dk=\int_0^{\infty}e^{-\mu u}  Q^{\delta}_z\left[\Sigma^{\delta}_{p,z,y} \in du \right],
\]
hence
\be \label{e:digital}
\int_0^u e^k
D(k;\delta,p,x,y)\, dk = Q^{\delta}_z\left[\Sigma^{\delta}_{p,z,y} \leq u\right], \qquad \forall u\geq 0.
\ee
 Moreover, once the function $D$ is determined by inverting the corresponding Laplace transform, one can also obtain the prices for put options written on the  accumulated interest. Indeed, for any $K>1$
\bea
P(\delta;z,y,
K)&:=&Q^{\delta}_z\left[\left(K-\exp\left(\Sigma^{\delta}_{p,z,y}\right)\right)^+\exp\left(-\Sigma^{\delta}_{p,z,y}\right)\right]
\nn \\
&=&
Q^{\delta}_z\left[\int_1^K \chf_{[\Sigma^{\delta}_{p,z,y} \leq \log
    k]}\,dk\exp\left(-\Sigma^{\delta}_{p,z,y}\right)\right]\nn \\
&=&\int_1^K  D(\log k;\delta,p,z,y)\,dk =\int_0^{\log K} e^u
D(u;\delta,p,z,y)\, du. \label{e:put}
\eea
Note that for $K=1$, i.e when the option is {\em at-the-money} since
the cumulative interest at $t=0$ is defined to be $1$, the put option is worthless. We can in fact get
how fast the option becomes worthless as $K$ approaches to
$1$. Indeed, comparing (\ref{e:put}) to (\ref{e:digital}) yields  the following
asymptotics for the option value in view of Theorem \ref{t:sbp}:
\be \label{e:smallKput}
\lim_{K \dar 1} \log K \log P(\delta;z,y, K)=-\frac{1}{2(p+1)^2}\left(z^{\frac{p+1}{2}}-y^{\frac{p+1}{2}}
\right)^2.
\ee
The above expression tells us how small the option price becomes when
the option is slightly {\em in-the-money}, i.e. when $K$ is very close
to $1$.

Next, assume $y < z$ and consider another type of a put option on the maximum of the short rate with maturity $R_y$ and payoff $(K-\max_{t \leq R_y}X_t)^+$ for $K>z$. The price of this option equals
\bean
Q^{\delta}_z\left[\exp\left(-\Sigma^{\delta}_{p,z,y}\right)\int_0^K \chf_{[\max_{t \leq R_y}X_t<a]}\,da\right]&=& Q^{\delta}_z\left[\exp\left(-\Sigma^{\delta}_{p,z,y}\right)\int_x^K\chf_{[R_a >R_y]}\,da \right]\\
&=&\frac{u_0(\sqrt{z})}{u_0(\sqrt{y})}\left(\frac{z}{y}\right)^{-\frac{\nu}{2}} \int_x^K\frac{\tilde{s_0}(z)-\tilde{s}(a)}{\tilde{s_0}(y)-\tilde{s_0}(a)}\,da,
\eean
in view of Corollary \ref{c:jmS} and Proposition \ref{p:jmS}, where
the pair $(u_0, \tilde{s_0})$ is computed by setting $\lambda=2$.

Finally, if one is interested in pricing an Asian option on the short
rate until time $R_y$, it suffices to use the identity
\[
Q^{\delta}_z\left[\left(K-\frac{\Sigma^{\delta}_{p,z,y}}{R_y}\right)^+\exp\left(-\Sigma^{\delta}_{p,z,y}\right)\right]=Q^{\delta}_z\left[\exp\left(-\Sigma^{\delta}_{p,z,y}\right)\int_0^K
      \chf_{[\Sigma^{\delta}_{p,z,y}< k R_y]}\,dk \right]
\]
and invert the joint Laplace transform of $R_y$ and
$\Sigma^{\delta}_{p,z,y} $ obtained in  Section \ref{s:2}.

\end{document}